\documentclass[a4paper, 11pt, oneside, onecolumn]{article}

\usepackage{hyperref}

\usepackage[T1]{fontenc} 
\usepackage[utf8]{inputenc}
\usepackage{ifthen}

\usepackage{titlesec}

\usepackage{amsthm}

\usepackage{graphicx}
\usepackage{caption}
\usepackage{geometry}
\usepackage[all]{xy}

\usepackage{enumerate}

\usepackage[obeyDraft, bordercolor=red!75!black, backgroundcolor=red!20!white, linecolor=red!75!black, textsize=footnotesize]{todonotes}

\usepackage{amssymb}
\usepackage{amsmath}
\usepackage{dsfont}
\usepackage{stmaryrd}
\usepackage{esint}

\usepackage{raccourcis}

\usepackage[backend=biber,style=alphabetic,maxnames=99]{biblatex}

\renewcommand{\theequation}{\arabic{section}.\arabic{equation}}

\newcounter{tout}[section]
\renewcommand{\thetout}{\arabic{section}.\arabic{tout}}

\theoremstyle{definition}
\newtheorem{definition}[tout]{Definition}

\newtheorem*{assumption}{Assumption}

\theoremstyle{plain}
\newtheorem{theorem}[tout]{Theorem}
\newtheorem{corollary}[tout]{Corollary}
\newtheorem{lemma}[tout]{Lemma}
\newtheorem{proposition}[tout]{Proposition}

\renewbibmacro{in:}{}
\AtEveryBibitem{%
  \clearlist{language}
}

\title{
\centering
Solvability of the Poisson-Dirichlet problem in domains with boundaries of mixed dimension}
\author{Léo Mandô\footnote{Laboratoire de mathématiques d'Orsay, Université Paris-Saclay, 307 rue Michel Magat, 91400 Orsay, France. Email: \textrm{leo.mando@universite-paris-saclay.fr}}}
\date{}

\begin{document}

\maketitle

\begin{abstract}
We extend several characterizations of the $L^p$-solvability of the homogeneous Dirichlet problem to (possibly degenerate) elliptic operators $L=-\textrm{div}(wA\nabla)$ defined on a large class of open sets $\Omega$ in $\mathbb{R}^n$. This framework encompasses not only uniformly elliptic operators in Lipschitz domains, but also Caffarelli-Sylvestre-type operators, and boundaries $\partial\Omega$ that are not $(n-1)$-dimensional, for instance. We prove that, for $1<p<\infty$, the $L^p$-solvability of the homogeneous Dirichlet problem is equivalent to the solvability of the Poisson-Dirichlet problem $Lu=wf-\textrm{div}(wF)$, \emph{i.e.} to the existence of a solution $u$ satisfying a $L^p$ non-tangential estimate whenever $f$ and $F$ belong to suitable weighted $L^p$ tent spaces. Furthermore, it is also equivalent to the solvability of the Poisson-Regularity problem for data in appropriate weighted $L^{p'}$ tent spaces, where estimates are obtained for $\nabla u$ rather than $u$.

\end{abstract}

\textbf{Keywords:} Degenerate elliptic operators, Boundary value problems, Dirichlet problem, Poisson-Dirichlet problem, tent spaces.

\tableofcontents

%%%%%%%%%%%%%%%%

\section{Introduction}

\subsection{Background}

In recent decades, the study of the (homogeneous) Dirichlet problem
\begin{equation}\label{eq:intro_dir}
\begin{cases}
Lu=0 &\texteq{in } \Omega\\
u=g &\texteq{on } \dOmega,
\end{cases}
\end{equation}
where $\Omega$ is an open set of $\R^n$ and $L=-\div(A\grad)$ is a uniformly elliptic operator in divergence form, has been an area of extensive research in harmonic analysis, pushing towards more general conditions: non-smooth coefficients $A$, rough boundaries $\dOmega$, non-continuous boundary data $g$. Our interest in this paper will lie in the solvability of the Dirichlet problem with $L^p$ data; that is to say, to find solutions of \eqref{eq:intro_dir} that satisfy the estimate
\begin{equation} \label{N<g}
\|\NN u\|_{L^p(\dOmega)}
\leq C\|g\|_{L^p(\dOmega)},
\end{equation}
where $\NN u$ is the non-tangential maximal function of $u$, see \eqref{eq:def_N} below for the definition.

When $L$ is the Laplacian and $\Omega$ is a Lipschitz domain, Dahlberg showed in \cite{DahlbergBjörn} that the harmonic measure is absolutely continuous with respect to the surface measure, and that the corresponding Radon derivative (called the Poisson kernel) belongs to the reverse H\"older class $B_2$. Putting this in our context, it means that in Lipschitz domains, the Dirichlet problem for the Laplacian is $L^2$-solvable.

Even 50 years after Dahlberg's paper, we do not know what is the geometric characterization of domains that ensures the $L^2$-solvability of the Dirichlet problem for the Laplacian.
What we know is that in Chord-Arc Domains (CAD) - a natural extension of Lipschitz domains - the harmonic measure is $A_\infty$-absolutely continuous\footnote{a quantitative and scale invariant version of absolute continuity, that will not be presented here but can be found in \cite{SteinElias}, Chapter V, 5.1.} with respect to the surface measure, which translates to the fact that the Dirichlet problem for the Laplacian is solvable only in $L^p$ for a large $p\in (1,\infty)$ depending of the CAD constants; indeed, the $A_\infty$-absolute continuity for CAD was tackled by David and Jerison in \cite{DAVID_lip}, while in \cite{Jer83} Jerison provided, for each $p\in (1,\infty)$, a CAD where the $L^p$-solvability of the Dirichlet problem fails. 
Chord Arc-Domains appears to be almost optimal for the existence of a $p\in (1,\infty)$ that gives the $L^p$ solvability of the Dirichlet problem for the Laplacian: assuming boundaries of Hausdorff dimension $n-1$ and ample access to the boundary (all in a quantitative manner), the converse was studied in \cite{HLMN17}, \cite{A21} and \cite{azzam_harmonic_2020}. % possibly, give more details on this last result, not mandatory.

Beyond the Laplacian, the $L^p$-solvability of the Dirichlet problem was widely studied. Early on, Caffarelli, Fabes, and Kenig in \cite{CFK}, and by Modica and Mortola in \cite{MM} exhibited some uniformly elliptic operators on smooth domains for which the $L^p$-solvability of the Dirichlet problem fails for every $1<p <\infty$. 
Positive results were subsequently obtained under additional assumptions on the coefficients: for symmetric coefficients in Lipschitz domains in \cite{JK81a}, for $t$-independent operators in $\R^{n}_+:= \{(x,t)\in\R^n \tq t>0\}$ in \cite{HKMP15}, for Dahlberg-Kenig-Pipher (DKP) operators - that are operators whose oscillations are controlled in terms of Carleson measures - in Lipschitz domains in \cite{KP01} and \cite{DPP01}. For DKP operators, the equivalence between CAD and the existence of some $p\in (1,\infty)$ that ensures the $L^p$-solvability of the Dirichlet was achieved in \cite{hofmann_uniform_2020}.

Let us also mention here a related boundary value problem: the $L^p$ regularity problem (or $\dot W^{1,p}$ Dirichlet problem). Instead of asking for \eqref{N<g}, we demand the bound
\begin{equation}\label{eq:regularity_problem}
\|\NNt(\grad_{\R^n} u)\|_{L^p(\dOmega)}
\leq C\|\grad_{\partial \Omega} g\|_{L^p}(\dOmega),
\end{equation}
where the meaning of $\nabla_{\partial \Omega}$, when the boundary is irregular, is part of the problem. The key point to understand is that there is a duality between the Dirichlet and the regularity problem: if the regularity problem is $L^p$-solvable for an operator $L$ in $\Omega$ (with an appropriate choice of gradient $\nabla_{\partial \Omega}$), then the Dirichlet problem is $L^{p'}$-solvable for the adjoint operator $L^*$, where $p'$ is the H\"older conjugate of $p$. The other implication - that is the Dirichlet problem implies the regularity problem - is false in general (see Proposition 1.9 in \cite{G25}) and require extra assumptions. It was established for the Laplacian in Lipschitz domains (\cite{JK81b}, \cite{Ver}), and recently for the Laplacian and DKP operators in domains with uniformly rectifiable boundaries (\cite{mourgoglou_regularity_2023} and  \cite{mourgoglou_solvability_2023}). 

To obtain those results on the Dirichlet and regularity problem, one relies on a large variety of characterization of the $L^p$-solvability of the Dirichlet problem. Let us mention one in particular, that will be the focus of our article and was a key ingredient of \cite{mourgoglou_solvability_2023}, which is the characterization of the $L^p$ (homogeneous) Dirichlet problem by its ``inhomogeneous'' counterpart.
The inhomogeneous Dirichlet problem is the study of the existence of a well-controlled solution to
\begin{equation}
\begin{cases}
Lu=f-\div F &\texteq{in } \Omega\\
u=g &\texteq{on } \dOmega,
\end{cases}
\end{equation}
for $g\in L^p(\dOmega)$ and an appropriate range of interior data $f$ and $F$. Roughly speaking - the missing definitions will be given later when we present our own results - Mourgoglou, Poggi and Tolsa proved in \cite{mourgoglou_solvability_2023} that for domains $\Omega$ with boundaries of dimension $n-1$, quantitatively, the $L^p$ solvability of the (homogeneous) Dirichlet problem for an uniformly elliptic operator $L$ in $\Omega$ is equivalent to the solvability of the Poisson-Dirichlet problem
\begin{equation} \label{PoissonDirichlet}
\begin{cases}
Lu=f-\div F &\texteq{in } \Omega\\
u=0 &\texteq{on } \dOmega,
\end{cases}
\end{equation}
in appropriate tent spaces, with a bound on $u$ defined as
\begin{equation}
\|\NNt u\|_{L^p(\partial \Omega)} \leq C\cro{\|\AAt(\delta^2f)\|_{L^p(\partial \Omega)} + \|\AAt(\delta F)\|_{L^p(\partial \Omega)} },
\end{equation}
where $\AAt$ is a modified $L^1$ area function defined in \eqref{eq:def_At}. By duality, the $L^p$-solvability of the (homogeneous) Dirichlet problem is also equivalent to the solvability of the Poisson-Dirichlet problem \eqref{PoissonDirichlet} for the adjoint operator $L^*$ with a bound on the gradient:
\begin{equation}
\|\NNt(\grad u)\|_{L^{p'}(\partial \Omega)} \leq C\cro{\|\AAt(\delta f)\|_{L^{p'}(\partial \Omega)} + \|\AAt(F)\|_{L^{p'}(\partial \Omega)}},
\end{equation}
which was named "solvability of the Poisson-regularity problem" by the authors of \cite{mourgoglou_solvability_2023}. Note that in the case of the unit ball, a slightly stronger version of the fact that the solvability of the Dirichlet problem implies the solvability of the Poisson-Dirichlet problem was already proven as an intermediate step of the main result in \cite{kenig_neumann_1995}. Finally, one should note that similar results have recently been established for inhomogeneous elliptic equations with other boundary conditions, in the context of Chord-Arc Domains: for Neumann boundary conditions, Feneuil and Li proved in \cite{FL24} that the solvability of the $L^p$- Poisson-Neumann problem implies the solvability of the (homogeneous) $L^\pp$-Neumann problem for the adjoint, and that the converse holds under the additional assumption that the $L^\pp$-Dirichlet problem is solvable. In \cite{FYY25}, Fu, Yang and Yang established the same result when replacing Neumann boundary conditions with Robin boundary conditions.

\medskip

In this paper, we generalize the equivalence between Dirichlet and Poisson-Dirichlet problem to a broader geometric setting. Indeed, most results concerning $L^p$ solvability have been restricted to domains with Ahlfors regular\footnote{a notion that characterizes the dimension of a set in a quantitative way, see \eqref{eq:Ahlfors}.} boundary of dimension  $n-1$ and uniformly elliptic operators, leaving out degenerate operators as studied by Fabes, Jerison, Kenig, and Serapioni in \cite{FJK83a}, \cite{FJK83b} and \cite{FKS}, the operators from the elliptic theory adapted to high co-dimensional boundaries developed by David, Feneuil, and Mayboroda in \cite{DFM21} and \cite{david_elliptic_2023}, and even the Caffarelli-Silvestre operator from \cite{CS07}. We choose to study the equivalence in a setting even more general that the one in \cite{david_elliptic_2023}, because we intends - as in \cite{mourgoglou_solvability_2023} - to remove the connectedness condition from our hypotheses on the domain, leading to many pitfalls due to the lack of comparison principle or change of pole for the elliptic measure.

For uniformly elliptic operators with degenerate coefficients, characterizations the $L^p$-solvability of the Dirichlet problem and the $A_\infty$-absolute continuity of the elliptic measure with respect to the surface measure are obtained in \cite{FP22} and \cite{cao_absolute_2022}. However, we will need to prove our new characterizations of the Dirichlet problem by adapting the characterizations of \cite{mourgoglou_regularity_2023} to our setting, since the characterizations of \cite{cao_absolute_2022} and \cite{FP22} relies on connectedness of the domain, which add an extra layer of difficulty to our article.

\subsection{General setting}

We consider an open set $\Omega\subset\R^n$, with $n\geq 2$, a Borel measure $m$ on $\R^n$, and a Borel measure $\mu$ supported on $\dOmega$ (both non-negative, non identically zero and locally finite). Our geometric setting is based on the one found in \cite{david_elliptic_2023}, although we do not assume the Harnack Chain condition like \cite{david_elliptic_2023} does, and instead we consider a measure $m$ on the whole space $\R^n$ (instead of a measure supported in $\Omega$) satisfying a Poincaré-Sobolev inequality on every ball centered in $\bOmega$.
\newline

Since our hypotheses on $\Omega$ are weaker than the ones from \cite{david_elliptic_2023}, our settings encompasses a large variety of domains and operators. The classical case, i.e. domains with $(n-1)$-Ahlfors regular boundaries and uniformly elliptic operators, is of course included. Let us recall for the sake of completeness that a set $E\incl\R^n$ is said to be \emph{$d$-Ahlfors regular} when there exists $C>0$ such that 
\begin{equation}\label{eq:Ahlfors}
C\inv r^d \leq \HH^d(E\cap B(x,r)) \leq Cr^d \for x\in E,\ 0<r<\diam(E),
\end{equation}
where $\HH^d$ is the $d$-dimensional Hausdorff measure.
If $\Omega = \R^n \setminus E$ is the complement of a $d$-Ahlfors regular set $E$ of dimension $d\leq n-2$, functions in $W^{1,2}(\Omega)$ do not have a trace, and one must consider degenerate operators like $L=-\div(\delta^{d-(n-1)}\grad)$, with $\delta(x):=\textrm{dist}(x,\dOmega)$, to be able to study the $L^p$-Dirichlet problem (see \cite{DFM19}, \cite{MZ19}, \cite{Fen22} or \cite{DM23} for examples of the study of the $A_\infty$-absolute continuity of the elliptic measure and the $L^p$ solvability of the Dirichlet problem in such settings). For an explanation of how our hypotheses encompasses the aforementioned cases, as well as others examples where our assumptions hold, see Chapter 3 in \cite{david_elliptic_2023}.

\begin{assumption}[H1]
There exists $C_{crk}\geq 1$ such that for any open ball $B=B(\xi,r)$ centered on $\dOmega$, we can find $x\in B$ such that
\begin{equation}\label{hyp:corkscrew}
B(x,C^{-1}_{crk} r)\subset\Omega.
\end{equation}
This assumption is commonly known as the corkscrew point condition.
\end{assumption}

\begin{assumption}[H2]
The measure $\mu$ is \emph{doubling}, i.e. there exists $C_\mu>1$ such that
\begin{equation}\label{hyp:mu_doubling}
\mu(2B)\leq C_\mu\mu(B)
\end{equation}
for any ball $B$ centered on $\dOmega$.
\end{assumption}

\begin{assumption}[H3]
$m_{|\Omega}$ is mutually absolutely continuous with respect to the Lebesgue measure: there exists a weight $w\in L^1\loc(\Omega)$ such that $w>0$ and $$dm(x)=w(x)dx.$$ In addition, $m$ is also doubling: there exists $C_m\geq 1$ such that
\begin{equation}\label{hyp:m_doubling}
m(2B) \leq C_m m(B)
\end{equation}
for any ball $B$ in $\R^n$.
\end{assumption}

\begin{lemma}
Assume (H1) and (H3). There exists  $C>0$, depending only on $C_{crk}$ and $C_m$, such that 
\begin{equation}\label{eq:mOmega_eqvt_m}
m(B\cap \Omega) \leq m(B)\leq C m(B\cap\Omega) \texteq{ for any } B=B(x,r)\texteq{centered in } \bOmega;
\end{equation}
consequently, $m_{|\Omega}$ is also doubling: there exists $C_m'>0$, depending only on $C_{crk}$ and $C_m$, such that
\begin{equation}\label{eq:mOmega_doubling}
m(2B\cap\Omega)\leq C_m'm(B\cap\Omega).
\end{equation}
\end{lemma}
\begin{proof}
The first inequality in \eqref{eq:mOmega_eqvt_m} is obvious; let us prove the second one. If $B/2\incl \Omega$, then by (H3) $m(B)\leq C_m m(B/2)\leq C_m m(B\cap \Omega)$. If $B/2$ is not included in $\Omega$, then there exists $\xi\in B/2\cap\dOmega$. The corkscrew condition (H1) give us $x'\in B(\xi,r/2)$ such that $B':=B\p{x',(2C_{crk})^{-1}r}\incl \Omega$. Then $B'\incl B \incl 4C_{crk} B'$, and by (H3) there exists $C>C_m$, depending on $C_{crk}$ and $C_m$, such that $m(4C_{crk} B')\leq C m(B')$. Altogether, this give us $m(B)\leq m(4C_{crk} B')\leq Cm(B')\leq Cm(B\cap\Omega)$. Finally, \eqref{eq:mOmega_doubling} is a direct consequence of (H3) and \eqref{eq:mOmega_eqvt_m}. 
\end{proof}

For any ball $B=B(x,r)$ centered in $\bOmega$ we define 
\begin{equation*}
\rho(B):= \frac{m(B\cap\Omega)}{r\mu(B\cap\dOmega)}.
\end{equation*}
We will also write $\rho(x):=\rho(B(x,2\delta(x)))$, for any $x\in \Omega$. When considering balls $B=B(\xi,r)$ centered on $\dOmega$, the quantity $\rho$ quantifies the deviation between two measures of $B\cap\Omega$, seen in one hand as a subset of $\Omega$ by using $m(B\cap\Omega)$, and in the other hand as a tent set over $B\cap\dOmega$ by using $r\mu(B\cap\dOmega)$. ``Standard'' deviations correspond to the case where $C^{-1} \leq \rho(B) \leq C$ for some $C>0$ whenever $B$ is a ball centered on $\partial \Omega$.
\newline

If we assume (H1), (H2) and (H3), then there exists $C$ depending only on $C_{crk}$, $C_\mu$ and $C_m$ such that for any ball $B$ centered on $\dOmega$,
\begin{equation}\label{eq:rho_doubl}
C^{-1}\rho(B) \leq \rho(2B)\leq C\rho(B).
\end{equation}
However, these bounds are not enough to guarantee some classic result of our theory, like the boundary Poincaré theorem, and we will need to assume a more precise bound on the growth of $\rho$.
\begin{assumption}[H4]
There exists $C_\rho\geq 1$ such that for any $B=B(\xi,r)$ centered on $\dOmega$, $\rho$ satisfies
\begin{equation}\label{hyp:rho_growth}
\rho(\Lambda B) \leq C_\rho\Lambda^{1-\epsilon}\rho(B) \for\Lambda>1,
\end{equation}
where $\epsilon=C_\rho\inv$.
\end{assumption}

\begin{assumption}[H5]\,
\begin{itemize}
    \item If $D$ is an open subset of $\R^n$ and $(\phi_k)$ is a sequence in $\CC^\infty(D)$ such that $\phi_k\to 0$ in $L^1(D,m)$, $\grad\phi_k\in L^2(D,m)^n$ for $k\in\N$ and $\grad \phi_k\to V$ in $L^2(D,m)^n$ for some $V\in L^2(D,m)^n$, then $V=0$.
    \item There exists $C_{pc}>0$ such that for any ball $B=B(x,r)$ centered in $\bOmega$,
    \begin{equation}\label{eq:poincare_hyp}
    \fint_B |\phi-\phi_B|\dm \leq C_{pc}r\pfrac{\fint_{2B}|\grad\phi|^2\dm},
    \end{equation}
    for any function $\phi\in \CC^\infty(2B)$ such that $\grad \phi\in L^2(2B,m)$, where $\phi_B:=\fint_B \phi\dm$.
\end{itemize}
\end{assumption}
Compared to the assumptions of \cite{david_elliptic_2023}, we removed the Harnack Chain condition, but assumed that $m$ is a doubling measure on the whole space $\R^n$ and that it satisfies a Poincaré inequality for all balls centered in $\bOmega$, instead of only assuming it for ball well contained in $\Omega$ as in \cite{david_elliptic_2023}. We adopt these stronger conditions mostly so that we can recover a boundary Poincaré inequality (Theorem \ref{th:boundary_poincare}), but it is possible that we could relax our hypotheses by only assuming an interior Poincaré inequality. However, since the article do not plan to dwell on capacity conditions and Poincar\'e inequalities, we decided to leave the answer to this question for a future project.
\newline

Finally, we consider uniformly elliptic operators ``with respect to the weight $w$'', that are operators in divergence form $L=-\div(wA\grad)$, with $A:\Omega\to M_n(\R)$ satisfying classical conditions of boundedness and ellipticity, i.e. there exists $C_A>0$ such that 
\begin{equation}\label{hyp:A_bounded}
A(x)V\cdot W \leq C_A |V||W|\for x\in\Omega \texteq{ and } V,W\in\R^n, 
\end{equation}
and
\begin{equation}\label{hyp:A_ellipitic}
A(x)V\cdot V \geq C_A\inv|V|^2\for x\in\Omega, V\in\R^n.
\end{equation}
We will write $L^*=-\div(wA^T\grad)$ the adjoint operator of $L$, where $A^T$ is the transpose of $A$.
\newline

The assumptions (H1), (H2), (H3), (H4) and (H5) are enough to have an elliptic theory (H\"older continuity of solutions, construction of Green functions and elliptic measure, and non-degeneracy of the elliptic measure) for uniformly elliptic operators with respect to $w$. The elliptic theory that we need is presented in Subsection \ref{SsPDprob}, and is proved in the book by Feneuil and Mayboroda currently in preparation (\cite{FM} %Analysis, Geometry, and PDE in a lower dimensional world. Joseph Feneuil and Svitlana Mayboroda, in preparation.
).

\subsection{Definitions an main result}

First, let us introduce some notation. In the rest of this work, we will write $a\siml b$ if there exists a constant $C>0$ such that $a\leq Cb$, and $a\simeq b$ if $a\siml b\siml a$. For $x\in\Omega$, we use $\delta(x):=\textnormal{dist}(x,\dOmega)$ and $B_x:=B\left(x,\delta(x)\right)$. For a given aperture $\alpha>0$ and a vertex $\xi\in\dOmega$, we define the cone 
\begin{equation*}
\gamma_\alpha(\xi):=\br{x\in\Omega\tq |x-\xi|<(1+\alpha)\delta(x)}.
\end{equation*}
We will use the \emph{non-tangential maximal function} $\NN^{(\alpha)}$ defined as
\begin{equation}\label{eq:def_N}
\NN^{(\alpha)} u(\xi):=\sup_{x\in\gamma(\xi)} |u(x)|,
\end{equation}
for $u$ $m$-measurable and $\xi\in\dOmega$. For functions that are not necessarily well defined everywhere, but only $L^2\loc$, we need a modified version of the non-tangential maximal function:
\begin{equation}\label{eq:def_Nt}
\NNt^{(\alpha,\lambda)} u(\xi):=\sup_{x\in\gamma_\alpha(\xi)} \pfrac{\fint_{\lambda B_x} |u|^2\dm},
\end{equation}
for $\alpha>0$, $0<\lambda\leq 1/2$, $u\in L^2\loc(\Omega,m)$ and $\xi\in\dOmega$. We also introduce the \emph{modified area function} $\AAt^{(\alpha,\lambda)}$:
\begin{equation}\label{eq:def_At}
\AAt^{(\alpha,\lambda)} f(\xi):=\int_{\gamma_\alpha(\xi)} \pfrac{\fint_{\lambda B_x} |f|^2\dm} \frac{\dm(x)}{m(B_x)},
\end{equation}
for $\alpha>0$, $0<\lambda\leq 1/2$, $f\in L^2\loc(\Omega,m)$ and $\xi\in\dOmega$. Note that the $L^p$ norms of these functions are equivalent for the different values of $\alpha>0$ and $0<\lambda \leq 1/2$, see Lemma \ref{lemma:eqvlce_NtAt}. Hence we write $\NN$ for $\NN^{(1)}$, $\NNt$ for $\NNt^{(1,1/2)}$ and $\AAt$ for $\AAt^{(1,1/2)}$ to lighten the notation. 
\newline

By Theorem \ref{th:elliptic_measure^} below, there exists a family of probability measures $(\omega^x)_{x\in\Omega}$ on $\dOmega$ - called elliptic measures associated to $L$ - such that for each $g\in C^0_c(\dOmega)$ the function $u$ defined by
\begin{equation}\label{eq:u_omega}
u(x)=\int_\dOmega g\domega^x \for x\in\Omega
\end{equation}
is a continuous weak solution to $Lu=0$ in $\Omega$ that extends continuously to $\partial \Omega$ by $u=g$ on $\partial \Omega$. We can now define what we mean by \emph{solvability of the Dirichlet problem with $L^p$ data} - or $L^p$ solvability.

\begin{definition}
Let $1<p<\infty$. We say that the Dirichlet problem for the operator $L$ with $L^p$ boundary data is solvable - abridged in $(D_p)$ - if there exists $C>0$ such that for every $g\in\CC_c(\dOmega)$ the solution $u$ to the Dirichlet problem
\begin{equation}
\begin{cases}
Lu=0&\texteq{in } \Omega\\
u=g & \texteq{on } \dOmega,
\end{cases}
\end{equation}
given by \eqref{eq:u_omega} satisfies the estimate 

\begin{equation}\label{eq:solvable_dirichlet}
\Lpd{\NN u} \leq C\Lpd{g}.
\end{equation}
\end{definition}

In Subsection \ref{ss:sob_spaces} we define $W_0(\Omega)$ as the completion of $\CC_c^\infty(\Omega)$ for the norm $\|\grad u\|_{L^2(\Omega,m)}$. For each $f\in L^\infty_c(\Omega,m)$ and $F\in L^\infty_c(\Omega,m)^n$, the Lax-Milgram theorem give us a unique weak solution in $W_0(\Omega)$ of $Lu=wf-\div(wF)$ (see Theorem \ref{th:solution_W}). We will use this notion of solution to define de solvability of the \emph{Poisson-Dirichlet problem}.
\begin{definition}
Let $1<p<\infty$. We say that $(PD_p)$ holds if there exists $C>0$ such that for each $f\in L^\infty_c(\Omega,m)$ and $F\in L^\infty_c(\Omega,m)^n$, the solution $u$ in $W_0(\Omega)$ to the Poisson-Dirichlet problem
\begin{equation}\label{eq:poisson_dirichlet}
\begin{cases}
Lu=wf-\div(wF)&\texteq{in } \Omega\\
u=0 & \texteq{on } \dOmega,
\end{cases}
\end{equation}
satisfies the estimate 
\begin{equation}\label{eq:solv_poisson_dirichlet}
\Lpd{\NNt u} \leq C\cro{\Lpd{\AAt(\delta^2f)} + \Lpd{\AAt(\delta F)} }.
\end{equation}
\end{definition}

We will also study the solvability of the Poisson-regularity problem for the adjoint operator, defined as follows.
\begin{definition}[Solvability of the Poisson-regularity problem]
Let $1<p<\infty$. We say that $(PR^*_\pp)$ holds if there exists $C>0$ such that for each $f\in L^\infty_c(\Omega,m)$ and $F\in L^\infty_c(\Omega,m)^n$, the solution $u$ in $W_0(\Omega)$ to the Poisson-Dirichlet problem
\begin{equation}
\begin{cases}
L^*u=wf-\div(wF)&\texteq{in } \Omega\\
u=0 & \texteq{on } \dOmega,
\end{cases}
\end{equation}
which satisfies the estimate 
\begin{equation}\label{eq:solv_poisson_reg}
\Lpd[\pp]{\NNt(\rho\grad u)} \leq C\cro{\Lpd[\pp]{\AAt(\rho\delta f)} + \Lpd[\pp]{\AAt(\rho F)}}
\end{equation}

\end{definition}

Note that we could also define the solvability estimates \eqref{eq:solv_poisson_dirichlet} and \eqref{eq:solv_poisson_reg} using a modified Carleson functional (like in \cite{mourgoglou_solvability_2023}), which in our setting would be defined as
\begin{equation}
\CCt^{(\lambda)} f(\xi) := \sup_{r>0} \frac{1}{\mu(B(x,r))} \int_{B(x,r)\cap\Omega} \pfrac{\fint_{\lambda B_x} |f|^2\dm} \frac{\dm(x)}{\rho(x)\delta(x)}
\end{equation}
for $0<\lambda\leq 1/2$, $f\in L^2\loc(\Omega,m)$ and $\xi\in\dOmega$. This is in fact equivalent, since $\Lpd{\CCt^{(\lambda)}f}\simeq\Lpd{\AAt f}$ when $1<p<\infty$ (see Theorem \ref{thA:eqvlce_CA} in the appendix).
\newline

We can now state our main result.
\begin{theorem}\label{th:eqvlce_DPD}
Assume \emph{(H1)-(H5)} are satisfied, and take $1<p<\infty$. The following are equivalent.
\begin{enumerate}[(i)]
\item $(D_p)$ holds.
\item $(PD_p)$ holds.
\item $(PD_p)$ holds whenever $F=0$.
\item $(PD_p)$ holds whenever $f=0$.
\item $(PR^*_\pp)$ holds.
\item $(PR^*_\pp)$ holds whenever $F=0$.
\end{enumerate}
\end{theorem}

Theorem \ref{th:eqvlce_DPD} shows the equivalence between the solvability of the (homogeneous) Dirichlet problem, the solvability of the Poisson-Dirichlet problem, and the solvability of the Poisson-regularity for the adjoint. Note that this result aligns for the most part with Theorem 1.22 in \cite{mourgoglou_solvability_2023}, except for the characterization (iii) which is new even in the setting of \cite{mourgoglou_solvability_2023}.

The article is divided as follows. In Section \ref{s:prelim}, %Preliminaries
we give the other definitions that are needed for our proofs, we provide basic elliptic theory results for our settings, and we finish by presenting results on tents spaces (that are proved in Appendix - Section \ref{s:appendix} - for completeness). In Section \ref{s:Dp}, %Solvability of (Dp)
we gave characterizations of the $L^p$ solvability of the Dirichlet problem in terms of either the elliptic measure or the Green function, that we shall rely on to prove Theorem \ref{th:eqvlce_DPD}. Section \ref{s:PDp} %Solvability of (PDp)
finishes the proof of Theorem \ref{th:eqvlce_DPD}, and extends the existence of solutions controlled as in \eqref{eq:solvable_dirichlet}, \eqref{eq:solv_poisson_dirichlet} or \eqref{eq:solv_poisson_reg} for general boundary and interior data.

%%%%%%%%%%%%%%%%%%%%%%%%%%%%%%%%%%%%%%%%%%%%%%%%%%%%%%%%%%%%%%%%%%%%%%%%%%
%Preliminaires
%%%%%%%%%%%%%%%%%%%%%%%%%%%%%%%%%%%%%%%%%%%%%%%%%%%%%%%%%%%%%%%%%%%%%%%%%%

\section{Preliminaries}\label{s:prelim}

\subsection{Sobolev spaces adapted to our setting}\label{ss:sob_spaces}

We present the spaces where the solutions lie. We follow the strategy and presentation \cite{david_elliptic_2023}.

\begin{definition}[Section 4 in \cite{david_elliptic_2023}]\label{def:W}
 Let us assume (H5), and let $D$ be an open subset of $\R^n$. We say that $u$ is in $W(D)$ if $u\in L^1\loc(D,m)$ and there exists $V\in L^2(D,m)^n$ and a sequence $(\phi_k)$ in $\CC^\infty(D)$ such that:
\begin{enumerate}[(i)]
    \item $\phi_k\to u$ in $L^1\loc(D,m)$,
    \item $\grad\phi_k$ is in $L^2(D,m)$ for any $k\in\N$,
    \item $\grad\phi_k\to V$ in $L^2(D,m)$.
\end{enumerate}
Observe that if $u\in W(D)$ the function $V$ in the definition is unique, thanks to (H5). We will write $\grad u$ this function. Then, we can equip the space $W(D)$ with the semi-norm $\|u\|_{W(D)}=\|\grad u \|_{L^2(D,m)}$.
\end{definition}
It is unclear whether (H5) is enough to guarantee that $\grad u$ corresponds to the distributional gradient of $u$. However, the two notion of gradient coincide on compactly supported Lipschitz functions (see Lemma 1.11 in \cite{HKM18}). We will also use a localized version of $W(D)$:
\begin{equation*}
W\loc(D)=\{u\in L^1\loc(D,m) \tq u\in W(\mathring K) \texteq{ for any compact set } K\incl D\}
\end{equation*}

The space $W(D)$ satisfies the following "completedness" property:
\begin{proposition}\label{prop:W_complete}
Let us assume $(H5)$, and let $D$ be an open subset of $\R^n$. If a sequence $(u_k)$ in $W(D)$ satisfies:
\begin{enumerate}[(i)]
\item $u_k$ is a Cauchy sequence in $L^1\loc(D,m)$,
\item $(\grad u_k)$ is a Cauchy sequence in $L^2(D,m)$,
\end{enumerate}
then there exists $u\in W(D,m)$ such that $u_k\to u$ in $L^1\loc(D,m)$ and $\grad u_k\to\grad u$ in $L^2(D,m)^n$.
\end{proposition}
\begin{proof}
Take $(u_k)$ in $W(D)$ satisfying $(i)$ and $(ii)$, and $j\in\N$. By completeness, there exists $u\in L^1\loc(D,m)$ and $V\in L^2(D,m)^n$ such that $u_k\to u$ in $L^1\loc(D,m)$ and $\grad u_k\to V$ in $L^2(D,m)$. Thus, there exists $k_j\in\N$ such that
\begin{equation*}
\|u-u_{k_j}\|_{L^1(K,m)}< 2^{-j}
\texteq{ and \;} \|V-\grad u_{k_j}\|_{L^2(D,m)}<2^{-j}
\texteq{ for any compact set } K\incl D;
\end{equation*}
by definition of $W(D)$, there exists $\phi_j\in\CC^\infty(D)$ such that $\grad\phi_j\in L^2(D,m)$ and
\begin{equation*}
\|u_{k_j}-\phi_j\|_{L^1(K,m)}< 2^{-j}
\texteq{ and \;} \|\grad u_{k_j}-\grad\phi_j\|_{L^2(D,m)}<2^{-j}
\texteq{ for any compact set } K\incl D.
\end{equation*}
Then we have that $\phi_j\to u$ in $L^1\loc(D,m)$ and $\grad\phi_j\to V$ in $L^2(D,m)$, wich implies  that $u\in W(D)$ with $\grad u=V$. This proves the result.
\end{proof}

The Poincaré-Sobolev inequality \eqref{eq:poincare_hyp} can be extended to all functions in $W(B)$, and self-improves to a better exponent in the left-hand side thanks to Theorem 1 in \cite{HK95}.

\begin{theorem}[Poincaré-Sobolev inequality]\label{th:poincare_int}
Assume \emph{(H3)} and \emph{(H5)}. There exists an exponent $2^*>2$ such that for any ball $B=B(x,r)$ in $\R^n$, any $0<\kappa\leq 1$ and any $u\in W(B)$, $u\in L^{2^*}(B,m)$ and
\begin{equation}\label{eq:poincare_int}
\pfrac[2^*]{\fint_B |u-u_E|^{2^*}\dm} \leq C_\kappa r\pfrac{\fint_B|\grad u|^2\dm},
\end{equation}
with $u_E:=\fint_E u\dm$, where $E$ is any Borel set in $B$ such that $m(E)>\kappa m(B)$ , and where $C_\kappa$ depends only on $C_m$, $C_{pc}$ and $\kappa$.
\end{theorem}
\begin{proof}
Take a ball $B$ in $\R^n$ and $u\in W(B)$. By definition we can approximate $u$ by smooth functions and use (H5) to obtain
\begin{equation*}
\fint_{B'} |u-u_{B'}|\dm
\leq C_{pc} r' \pfrac{\fint_{2B'} |\grad u|^2\dm},
\end{equation*}
for any ball $B'=B(x',r')$ such that $2B'\incl B$. Then, Theorem 1 in \cite{HK95} give us that $u\in L^{2^*}(B,m)$ and \eqref{eq:poincare_int} with $E=B$. Now, take $0<\kappa\leq 1$ and a Borel set $E$ in $B$ such that $m(E)>\kappa m(B)$. We have
\begin{align*}
\pfrac[2^*]{\fint_{B} |u-u_E|^{2^*}\dm}
&\leq \pfrac[2^*]{\fint_{B} |u-u_B|^{2^*}\dm} + |u_B-u_E|\\
&\leq \pfrac[2^*]{\fint_{B} |u-u_B|^{2^*}\dm} + \fint_E |u-u_B|\dm\\
&\siml \pfrac[2^*]{\fint_{B} |u-u_B|^{2^*}\dm}
\siml r\pfrac{\fint_B |\grad u|^2\dm},
\end{align*}
where we used the fact that $E\incl B$ with $m(E)\simeq m(B)$ and the Hölder inequality for the third estimate, and \eqref{eq:poincare_int} with $E=B$ for the last one.
\end{proof}

We will now define a set of function $u\in W(D\cap\Omega)$ such that, in some sense, "$u$ is null on $D\cap\dOmega$".
\begin{definition}[$W_0(D\cap\Omega)$]\label{def:W_0}
Let us assume (H5), and let $D$ be an open subset of $\R^n$. We say that $u$ is in $W_0(D\cap\Omega)$ if $u\in L^1\loc(D\cap\Omega,m)$ and there exists $V\in L^2(D\cap\Omega,m)^n$ and a sequence $(\phi_k)$ in $\CC^\infty(D\cap\Omega)$ such that:
\begin{enumerate}[(i)]
	\item $\supp \phi_k\incl \Omega$ for any $k\in\N$, 
    \item $\phi_k\to u$ in $L^1\loc(D\cap\Omega,m)$,
    \item $\grad\phi_k$ is in $L^2(D\cap\Omega,m)$ for any $k\in\N$,
    \item $\grad\phi_k\to V$ in $L^2(D\cap\Omega,m)$.
\end{enumerate}
\end{definition}
We highlight the fact that the notation $W_0(D\cap\Omega)$ denote a space of functions that are null on $D\cap\dOmega$, and not on $\partial (D\cap\Omega)$. These functions satisfy a Poincaré inequality:
\begin{theorem}[Boundary Poincaré inequality]\label{th:boundary_poincare}
Assume \emph{(H1)-(H5)}. For any ball $B=B(\xi,r)$ centered on $\dOmega$, we have:
\begin{equation}\label{eq:boundary_poincare}
\pfrac[2^*]{\fint_{B\cap\Omega} |u|^{2^*}\dm }
\leq Cr \pfrac{\fint_{B\cap\Omega} |\grad u|^2\dm},
\end{equation}
whenever $u\in W_0(B\cap\Omega)$, where $C$ depends only on $C_{crk}$, $C_\mu$, $C_m$, $C_\rho$, and $C_{pc}$.
\end{theorem}

Before beginning the proof of the boundary Poincaré inequality, let us state a small technical lemma.
\begin{lemma}
Assume \emph{(H1)-(H4)}. For any ball $B=B(\xi_0,r)$ centered on $\dOmega$ and any $g\in L^2(2B,m)$ such that $g\geq 0$,
\begin{equation}\label{eq:fubini_decr}
\fint_{B\cap\dOmega} \pfrac{\fint_{B(\xi,\lambda r)\cap\Omega} g^2\dm}\dmu(\xi)
\leq C\lambda^{\epsilon/2-1}\pfrac{\fint_{2B\cap\Omega} g^2\dm},
\for 0<\lambda\leq 1,
\end{equation}
where $\epsilon=C_\rho^{-1}$  as in assumption \emph{(H4)}, and where $C$ depends only on $C_{crk}$, $C_\mu$, $C_m$ and $C_\rho$.
\end{lemma}
\begin{proof}
Take a ball $B=B(x_0,r)$ centered on $\dOmega$, $g\in L^p(2B)$ and $0<\lambda\leq 1$. For any $\xi\in B\cap\dOmega$,
\begin{align}\label{eq:fubini_decr:rho}
\pfrac{\fint_{B(\xi,\lambda	r)\cap\Omega} g^2\dm}
&=\p{\lambda r\rho(B(\xi,\lambda r))}^{-1/2}
\pfrac{\frac{1}{\mu(B(\xi,\lambda r))} \int_{B(\xi,\lambda r)\cap\Omega}g^2\dm }\nonumber\\
&\siml \lambda^{\epsilon/2-1}(r\rho(B))^{-1/2} \pfrac{\frac{1}{\mu(B(\xi,\lambda r))} \int_{B(\xi,\lambda r)\cap\Omega}g^2\dm },
\end{align}
using \eqref{eq:rho_doubl} and \eqref{hyp:rho_growth}, where $\epsilon>0$ is given in assumption (H4). Furthermore,
\begin{align}\label{eq:fubini_decr:fub}
\fint_{B\cap\dOmega}\pfrac{\frac{1}{\mu(B(\xi,\lambda r))} \int_{B(\xi,\lambda r)}g^2\dm }&\dmu(\xi)
\leq  \pfrac{\fint_{B\cap\dOmega}\frac{1}{\mu(B(\xi,\lambda r))} \int_{B(\xi,\lambda r)}g^2\dm \dmu(\xi)}\nonumber\\
&\leq\pfrac{ \frac{1}{\mu(B)} \int_{2B\cap\Omega} g(x)^2\int_{B(x,\lambda r)\cap\dOmega} \frac{\dmu(\xi)}{\mu(B(\xi,\lambda r))} \dm(x)}\nonumber\\
&\siml \pfrac{ \frac{1}{\mu(B)}\int_{2B} g^2\dm },
\end{align}
where we used Hölder for the first estimate, Fubini and the fact that $B(\xi,\lambda r)\incl 2B$ for the second one, and the fact that $B(x,\lambda r)\incl B(\xi,2\lambda r)$ and the doubling property of $\mu$ for the third one. Piecing together \eqref{eq:fubini_decr:rho} and \eqref{eq:fubini_decr:fub}, we obtain that 
\begin{align*}
\fint_{B\cap\dOmega} \pfrac{\fint_{B(\xi,\lambda r)\cap\Omega} g^2\dm}\dmu(\xi)
&\siml \lambda^{\epsilon/2-1} \p{r\rho(B)\mu(B)}^{-1/2} \pfrac{\int_{2B}g^2\dm}\\
&\siml\lambda^{\epsilon/2-1} \pfrac{\fint_{2B}g^2\dm},
\end{align*}
which is the desired result.
\end{proof}
Now let us prove the boundary Poincaré inequality.

\begin{proof}[Proof of Theorem \ref{th:boundary_poincare}]
Take a ball $B=B(\xi_0,r)$ centered on $\dOmega$ and a function $u\in W_0(B\cap\Omega)$. By definition of $W_0(B\cap\Omega)$, it suffices to prove the theorem assuming that $u\in \CC^\infty(B\cap\Omega)$ and that $\supp u\incl \Omega$ to finish the proof. Then, we can extend $u$ by $0$ to $B\priv\Omega$, giving us a function in $\CC^\infty(B)$ that we will still write $u$ in order to lighten the notation. Observe that the gradient of the extension satisfies $\grad u=0$ in $B\priv \Omega$. Now, let us write $B_k^\xi:=B(\xi,2^{-k}r)$ and $B^\xi_0=B$, for any $\xi\in B$ and $k\in\N^*$.
\newline

First, we have that  
\begin{equation*}
\pfrac[2^*]{\fint_{B\cap\Omega} |u|^{2^*}\dm}
\leq \pfrac[2^*]{\fint_{B} |u-u_B|^{2^*}\dm} + |u_B|
\siml  r\pfrac{\fint_B |\grad u|^2\dm} + |u_B|,
\end{equation*}
using the Poincaré-Sobolev inequality \eqref{eq:poincare_int}. Thus, it suffices to prove that 
\begin{equation}\label{eq:boundary_poincare:est_avg}
|u_B|\siml r\pfrac{\fint_B|\grad u|^2\dm}.
\end{equation}
Since $\supp u\incl\Omega$, we have that $u_{B^\xi_k}\to 0$ for all $\xi\in B/2\cap\dOmega$; thus, 
\begin{equation*}
|u_B|
\leq \sum_{k\in\N} \abs{u_{B^\xi_k}-u_{B^\xi_{k+1}}}
\leq \sum_{k\in\N} \fint_{B^\xi_k} \abs{u-u_{B^\xi_{k+1}}}\dm
\siml \sum_{k\in\N} 2^{-k}r\pfrac{\fint_{B^\xi_k}|\grad u|^2\dm},
\end{equation*}
using the Poincaré-Sobolev inequality \eqref{eq:poincare_int}. Then, averaging in $\xi$ over $B/2$:
\begin{align*}
|u_B|
\siml \fint_{B/2} \sum_{k\in\N} 2^{-k}r\pfrac{\fint_{B^\xi_k}|\grad u|^2\dm} \dmu(\xi)
&=r\sum_{k\in\N} 2^{-k}\fint_{B/2}\pfrac{\fint_{B^\xi_k}|\grad u|^2\dm} \dmu(\xi)\\
&\siml r\sum_{k\in\N} 2^{-\epsilon k/2} \pfrac{\fint_B |\grad u|^2\dm}\\
&\siml r\pfrac{\fint_B |\grad u|^2\dm},
\end{align*}
where we used Fubini for the second estimate and \eqref{eq:fubini_decr} for the third one. We have proven \eqref{eq:boundary_poincare:est_avg}, which concludes the proof.
\end{proof}

As a consequence, we have the following convergence result for functions in $W_0(D\cap\Omega)$:
\begin{corollary}
Assume \emph{(H1)-(H5)}, and let $D$ be an open subset of $\R^n$. Any function $u\in W_0(D\cap\Omega)$ satisfy
\begin{equation}\label{eq:cv_W0}
\pfrac[2^*]{\fint_{B(\xi,r)\cap D\cap\Omega} |u|^{2^*}\dm} \cv{r\to 0} 0 \formuae \xi\in D\cap\dOmega;
\end{equation}
in particular, we have the non-tangential convergence
\begin{equation}\label{eq:cv_nt_W0}
\fint_{B_x/2\cap D} u\dm \cvstack{x\to\xi}{x\in\gamma(\xi)} 0 \formuae \xi\in D\cap\dOmega. 
\end{equation}
\end{corollary}
\begin{proof}
Take $u\in W_0(D)$ and $B=B(\xi_0,r)$ centered on $D\cap\dOmega$ such that $2B\incl D$. Then for any $0<\lambda<1$ and any $\xi\in B\cap\dOmega$, $B(\xi,\lambda r)\incl D$ and 
\begin{align*}
\fint_B\pfrac[2^*]{\fint_{B(\xi,\lambda r)\cap\Omega} |u|^{2^*}\dm}\dmu(\xi)
&\siml \lambda r\fint_B\pfrac{\fint_{B(\xi,\lambda r)\cap\Omega} |\grad u|^2\dm}\dmu(\xi)\\
&\siml r\lambda^{\epsilon/2} \pfrac{\fint_{2B} |\grad u|^2\dm},
\end{align*}
where we used the boundary Poincaré inequality \eqref{eq:boundary_poincare} for the first estimate and \eqref{eq:fubini_decr} for the second one. In particular,
\begin{equation*}
\pfrac[2^*]{\fint_{B(\xi,\lambda r)\cap D\cap\Omega} |u|^{2^*}\dm}
=\pfrac[2^*]{\fint_{B(\xi,\lambda r)\cap\Omega} |u|^{2^*}\dm}
\cv{\lambda\to 0} 0 \formuae \xi\in B.
\end{equation*}
Since $D\cap\Omega$ can be covered by a countable number of such balls $B$, we obtain \eqref{eq:cv_W0}. The non-tangential convergence is weaker, and is a direct consequence of \eqref{eq:cv_W0}: indeed, for any $\xi\in D\cap\dOmega$ and $x\in\gamma(\xi)$, if $\delta(x)$ is small enough then $B(\xi,4\delta(x))\incl D$, and $B_x/2\incl B(\xi,4\delta(x))\incl 8B_x$; using these facts and the doubling property of $m_{|\Omega}$ (see \eqref{eq:mOmega_doubling}), we can deduce \eqref{eq:cv_nt_W0} from \eqref{eq:cv_W0}.
\end{proof}

The boundary Poincare inequality implies that $W_0(\Omega)$ is complete.
\begin{theorem}\label{th:W0_complete}
Let us assume \emph{(H1)-(H5)}. Then $\|\cdot\|_{W(\Omega)}$ defines a norm on $W_0$, and $W_0$ equipped with the associated scalar product $\langle u,v \rangle_{W(\Omega)}:=\int_\Omega \grad u\cdot\grad v\dm$ is a Hilbert space.
\end{theorem}
\begin{proof}
First, let us prove that the semi-norm $\|\cdot\|_{W(\Omega)}$ is a norm on $W_0$. If $u\in W_0(\Omega)$ satisfies $\|u\|_{W(\Omega)}=0$, then for any $B=B(\xi,r)$ centered on $\dOmega$  the boundary Poincaré inequality \eqref{eq:boundary_poincare} give us
\begin{equation*}
\fint_{B\cap\Omega} |u|\dm\siml r\pfrac{\fint_{B\cap\Omega} |\grad u|^2\dm} =0,
\end{equation*}
which implies that $u=0$ a.e. in $B\cap\Omega$. Since $\Omega$ can be covered by a countable number of such balls, we have $u=0$ a.e. on $\Omega$.
\newline

Now, let us prove that $W_0(\Omega)$ is a Hilbert space, and take a Cauchy sequence $(u_k)$ in $W_0(\Omega)$. By the boundary Poincaré inequality, for any $B=B(\xi,r)$ centered on $\dOmega$:
\begin{equation*}
\fint_{B\cap\Omega} |u_k-u_\ell|\dm
\siml r\pfrac{\fint_{B\cap\Omega} |\grad (u_k-u_\ell)|^2\dm} \for k,\ell\in \N,
\end{equation*}
implying that $(u_k)$ is a Cauchy sequence in $L^1\loc(\Omega,m)$. Since $(\grad u_k)$ is a Cauchy sequence in $L^2(\Omega,m)$ by assumption, we can use the definition of $W_0$ and reproduce the proof of Proposition \ref{prop:W_complete} to show that there exists $u\in W_0(\Omega)$ such that $u_k\to u$ in $L^1\loc(\Omega,m)$ and $\grad u_k\to \grad u$ in $L^2(\Omega,m)^n$, which proves the completeness of $W_0(\Omega)$.
\end{proof}

\subsection{The Poisson-Dirichlet problem} \label{SsPDprob}

%All results in this subsection can be found in \cite{david_elliptic_2023}, although they assume quantitative connectedness. Nevertheless, most of the proofs still hold without connectedness, and we give alternative proofs for the one that does not. An article establishing an elliptic theory in non-connected sets will be published by Joseph Feneuil and Svitlana Mayboroda.
In this subsection, we always assume that (H1)-(H5) hold.

\subsubsection{Weak solutions and their first properties}

\begin{definition}\label{def:weak_solution}
Take $D$ an open set in $\Omega$, $f\in L^1\loc(D)$ and $F\in L^1\loc(D,m)^n$. We say that $u\in W\loc(D)$ is a (weak) solution to $Lu=wf-\div(wF)$ in $D$ when
\begin{equation}
\int_D A\grad u\cdot\grad\phi \dm
=\int_D f\phi\dm + \int_D F\cdot\grad\phi\dm \for \phi\in\CC^\infty_c(D)
\end{equation}
\end{definition}

In the case of homogeneous boundary condition, we have existence and uniqueness of the solutions of the Poisson-Dirichlet problem
\begin{equation*}
\begin{cases}
Lu=wf-\div(wF) &\texteq{in }\Omega\\
u=0 &\texteq{on } \dOmega
\end{cases}
\end{equation*}
in the following sense.

\begin{theorem}\label{th:solution_W}
Let us write $W_0(\Omega)'$ the vector space of continuous linear forms on $W_0(\Omega)$, equipped with its usual norm. For any $\ell\in W_0(\Omega)'$, there exists a unique $u\in W_0(\Omega)$ such that 
\begin{equation}\label{eq:solution_faible_forme_lineaire}
\int_\Omega A\grad u \cdot\grad v \dm = \ell(v),\for v\in W_0.
\end{equation}
In particular, there exists a unique solution in $W_0(\Omega)$ of $Lu=wf-\div(wF)$ whenever $f\in L^{2_*}_c(\bOmega,m)$ and $F\in L^2(\Omega,m)^n$, where we write $2_*$ the dual exponent of the sobolev exponent $2^*$ from Theorem \ref{th:poincare_int}.
\end{theorem}
\begin{proof}
The first statement is a classic result: it relies on the Lax-Milgram theorem, using the fact that $W_0(\Omega)$ is complete (Theorem \ref{th:W0_complete}). The second statement is a direct application of the first one to the linear form defined by
\begin{equation*}
\ell(v):=\int_\Omega fv\dm
+ \int_\Omega F\cdot\grad v\dm \for v\in W_0,
\end{equation*}
which is continuous thanks to Hölder's inequality and the boundary Poincaré inequality \ref{eq:boundary_poincare}.
\end{proof}

We will now present some classical \emph{a priori} results for solutions associated to the elliptic operator $L$. Let us begin with interior results: these are proven in a way that matches our setting in \cite{david_elliptic_2023}.
\begin{proposition}[Interior Caccioppoli inequality]
For every ball $B=B(x,r)$ satisfying $2B\incl\Omega$ and every $u \in W(2B)$ weak solution of $Lu=wf-\div(wF)$ in $2B$ with $f\in L^{2_*}(2B)$ and $F\in L^2(2B)^n$,  
\begin{equation}\label{eq:int_caccio}
\pfrac{\fint_B |\grad u|^2\dm} 
\leq \frac{C}{r}\pfrac{\fint_{2B}|u|^2\dm} 
+ Cr\pfrac[2_*]{\fint_{2B}|f|^{2_*}\dm}
+ C\pfrac{\fint_{2B} |F|^2\dm},
\end{equation}
where $C>0$ depends only on $C_m$, $C_{pc}$ and $C_A$.
\end{proposition}

\begin{proof}
Let us consider $\chi\in\CC^\infty_c(\R^n)$ such that $\UN_B\leq\chi\leq\UN_{2B}$ and $\|\grad\chi\|_\infty\siml\frac{1}{r}$. Let us write $v=\chi^2u$; by definition $v$ is in $W$, and its support is in $2B$, so in particular $v\in W_0$.   Since $u$ is a weak solution of $Lu=wf-\div(wF)$, we have
\begin{align*}
\int_\Omega A\grad u\cdot\grad v \dm
&=\int_\Omega fv\dm + \int_\Omega F\cdot\grad v\dm\\
&\leq \pfrac[2_*]{\int_{2B} |f|^{2_*}\dm} \pfrac[2^*]{\int_{2B} |v|^{2^*}\dm} + \pfrac{\int_{2B} |F|^2\dm }\pfrac{\int_{2B} |\grad v|^2\dm }\\
&\leq N(B,f,F) \pfrac{\int_\Omega|\grad v|^2\dm },
\end{align*}
using the Poincaré-Sobolev inequality \eqref{eq:poincare_int} together with the fact that $\supp v \subset 2B$, and where
\begin{align*}
N(B,f,F)\simeq rm(B)^{1/2}\pfrac[2_*]{\fint_{2B} |f|^{2_*}\dm} + \pfrac{\int_{2B} |F|^2\dm }.
\end{align*} 
Since $\grad v=\chi^2\grad u + 2u\grad\chi$, we obtain
\begin{equation*}
\int_\Omega \chi^2 A\grad u\cdot\grad u \dm
\leq -2\int_\Omega \chi u A\grad u\grad\chi\dm + N(B,f,F)\cro{\pfrac{\int_\Omega |\chi\grad u|^2\dm} + \pfrac{\int_\Omega |\chi u\grad\chi\dm|^2 }}.
\end{equation*}
Then, using the boundedness \eqref{hyp:A_bounded} and ellipticity \eqref{hyp:A_ellipitic} of $A$, 
\begin{align*}
\int_\Omega |\chi\grad u|^2\dm 
&\siml\int_\Omega |u\grad\chi||\chi\grad u|\dm + N(B,f,F)\cro{\pfrac{\int_\Omega |\chi\grad u|^2\dm} + \pfrac{\int_\Omega |u\grad\chi|^2\dm }}\\
&\siml \pfrac{\int_\Omega |u\grad\chi|^2\dm }\pfrac{\int_\Omega |\chi\grad u|^2\dm}\\
&+ N(B,f,F)\cro{\pfrac{\int_\Omega |\chi\grad u|^2\dm} + \pfrac{\int_\Omega |u\grad\chi|^2\dm }}\\
&\leq \frac{1}{2}\int_\Omega |\chi\grad u|^2\dm + C\int_\Omega |u\grad\chi|^2\dm + CN(B,f,F)^2,
\end{align*}
where we used (three times) that for any $\epsilon>0$ there exists $C>0$ such that for any $a,b$ in $\R$, $ab\leq \epsilon a^2 + Cb^2$. We finally get that
\begin{equation*}
\int_B |\grad u|^2\dm 
\leq \int_\Omega |\chi\grad u|^2\dm
\siml \int_{2B}|u\grad \chi|^2\dm  + N(B,f,F)^2
\siml \frac{1}{r^2}\int_{2B}|u|^2\dm  + N(B,f,F)^2,
\end{equation*}
which implies the desired result.
\end{proof}

\begin{proposition}[Interior Moser estimate, Lemma 11.18 in \cite{david_elliptic_2023}]
Take $0<p<\infty$. For every ball $B=B(x,r)$ satisfying $2B\incl\Omega$ and every $u \in W(2B)$ weak solution of $Lu=0$ in $2B$, 
\begin{equation}\label{eq:int_moser}
\sup_B |u|\leq C\pfrac[p]{\fint_{2B} |u|^p\dm},
\end{equation}
where $C>0$ depends only on $n, C_m, C_{pc}, C_A$ and $p$. 
\end{proposition}

\begin{proposition}[Harnack inequality, Lemma 11.35 in \cite{david_elliptic_2023}]
For every ball $B=B(x,r)$ satisfying $2B\incl\Omega$ and every $u\geq 0$ weak solution of $Lu=0$ in $2B$, 
\begin{equation}\label{eq:harnack}
\sup_B u \leq C\inf_B u.
\end{equation}
where $C>0$ depends only on $n, C_m, C_{pc}$ and $C_A$.
\end{proposition}

Boundary results are more delicate, and to the best of our knowledge, there is no work in the literature that perfectly matches our setting. The article \cite{david_elliptic_2023} assumes a quantitative connectedness condition on the domain (the Harnack chain condition). Yet, for the Hölder continuity of solutions and the properties of the Green function, their arguments rely only on a boundary Poincaré inequality, which in our case is provided by Theorem \ref{th:boundary_poincare} above. As for the book \cite{HKM06}, the authors assume that bounded domains. Nonetheless, we shall cite our results by refering to either \cite{david_elliptic_2023} or \cite{HKM06}, and the reader will be able to find the results in our exact in the book \cite{FM} %To complete
under preparation.
\newline

The following result can be found either in \cite{david_elliptic_2023}, Lemma 11.28, or in \cite{HKM06}, Theorem 6.44.
\begin{proposition}[Boundary oscillation estimate]
For any function $f$ defined on a set $E$ of $\R^n$, we define its oscillation on $E$ as
\begin{equation*}
\osc{E} f=\sup_{x,y\in E} |f(x)-f(y)|.
\end{equation*}
There exists $0<\epsilon<1$ such that for every ball $B$ centered on $\dOmega$ and every $u\in W_0(2B\cap\Omega)$ solution of $Lu=0$ on $2B\cap\Omega$,
\begin{equation}\label{eq:osc_estimate}
\osc{B\cap\Omega} u \leq (1-\epsilon)\osc{2B\cap\Omega} u.
\end{equation}
The constant $\epsilon$ depends only on $n, C_{crk}, C_\mu, C_m, C_\rho, C_{pc}$ and $C_A$.
\end{proposition}

\begin{proposition}[Boundary Hölder continuity, Lemma 11.32 in \cite{david_elliptic_2023}]
There exists $0<\eta\leq 1$ such that for every ball $B=B(\xi,r)$, centered on $\dOmega$ and every $u\in W_0(B\cap\Omega)$ solution of $Lu=0$ in $B\cap\Omega$, $u$ is $\eta$-Hölder continuous in $B\cap\bOmega$ and
\begin{equation}\label{eq:boundary_holder}
\sup_{\lambda B\cap\Omega} |u|
\leq C \lambda^\eta \pfrac{\fint_{B\cap\Omega} |u|^2\dm} \for 0<\lambda< 1/2.
\end{equation}
The constants $\eta$ and $C$ depend only on $n, C_{crk}, C_\mu, C_m, C_\rho, C_{pc}$ and $C_A$.

\end{proposition}

\subsubsection{Green function}

The solution to the Poisson-Dirichlet problem
\begin{equation*}
\begin{cases}
Lu=wf -\div(wF) &\texteq{in } \Omega\\
u=0 &\texteq{on } \dOmega
\end{cases}
\end{equation*}
can be represented thanks to an integral kernel: the Green function (see Theorem \ref{th:Green_repr}). A construction can be found in Chapter 14 of \cite{david_elliptic_2023}.

\begin{theorem}[Theorem 14.60 in \cite{david_elliptic_2023}]\label{th:Green}
There exists a function $G:\Omega\times\Omega\to [0,\infty]$ with the following properties:
\begin{enumerate}[(i)]
\item For any $y\in\Omega$ and any $\chi\in\CC_c^\infty(\R^n)$ such that $\chi=1$ in a neighborhood of $y$, $(1-\chi)G(\cdot,y)\in W_0(\Omega)$; in particular, $G(\cdot,y)\in W_0(\Omega \setminus \lambda B_y)$ for any $0<\lambda\leq 1/2$.
\item For $\phi\in\CC^\infty_c(\Omega)$ and $y\in\Omega$, 
\begin{equation*}
\int_\Omega A\grad_x G(x,y)\grad\phi(x)\dm(x)=\phi(y).
\end{equation*}
In particular, $G(\cdot,y)$ is a solution of $Lu=0$ in $\Omega\priv\{y\}$, and $G$ is locally Hölder continuous in $\Omega\priv\br{y}$.

\item For any $0<\lambda\leq 1/2$ and any $x,y\in\Omega$ satisfying $|x-y|\geq \lambda\delta(x)$, 
\begin{equation}\label{eq:majo_green}
G(x,y)\leq C_\lambda \frac{|x-y|^2}{m(B(x,|x-y|)\cap\Omega)},
\end{equation}
where $C_\lambda$ depends only on $n, C_{crk}, C_\mu, C_m, C_\rho, C_{pc}, C_A$ and $\lambda$.

\end{enumerate}
\end{theorem}
Let us write $L^*=-\div(wA^T\grad)$, where $A^T$ is the transpose of $A$. Since $A^T$ shares the same properties as $A$, there exists $G^*$ given by Theorem \ref{th:Green} with the same properties as $G$, but with $A^T$ instead of $A$.
\begin{proposition}[Lemma 14.78 in \cite{david_elliptic_2023}]
With the notation above:
\begin{equation*}
G^*(x,y)=G(y,x)\for x,y\in\Omega.
\end{equation*}
In particular, the functions $G(y,\cdot)$ for $y\in\Omega$ satisfies estimate $(iii)$ of Theorem \ref{th:Green}.
\end{proposition}

We have the following representation theorem.
\begin{theorem}\label{th:Green_repr}
$G$ is the unique function from $\Omega\times\Omega$ to $[0,\infty]$ such that $G(x,\cdot)$ is in $L^1\loc(\Omega,m)$ for every $x\in\Omega$, and  such that for any $f\in L^\infty_c(\Omega,m)$ and $F\in L^\infty(\Omega,m)^n$ the function $u$ given by
\begin{equation*}\label{eq:repr_green}
u(x)=\int_\Omega G(x,y)f(y)\dm(y)
+\int_\Omega \grad_y G(x,y)\cdot F(y)\dm(y) \for x\in\Omega,
\end{equation*}
is the solution in $W_0(\Omega)$ of $Lu=wf-\div(wF)$ in $\Omega$ given by Theorem \ref{th:solution_W}. 
\end{theorem}

\begin{proof}
The proof is like Lemma 10.7 in \cite{DFM21} once we can prove that a weak solution to $Lu= wf-\div(wF)$ in $\Omega$ is continuous inside the domain. The later can be proved with Moser iterations, see the book \cite{FM} % The book in preparation
in preparation.
\end{proof}

\subsubsection{Non-homogeneous boundary conditions and elliptic measure}

We now turn to the study of the Dirichlet problem
\begin{equation}\label{eq:dir_pb}
\begin{cases}
Lu=0 &\texteq{in } \Omega\\
u=g &\texteq{on } \dOmega.
\end{cases}
\end{equation}

If $g$ is regular enough, Theorem \ref{th:solution_W} a unique solution to \eqref{eq:dir_pb} in the following sense.

\begin{proposition}\label{prop:solution_W_chi}
For any $\chi\in\CC^\infty_c(\R^n)$ there exists a unique solution in $W(\Omega)$ of $Lu=0$ satisfying $u-\chi\in W_0(\Omega)$. 
\end{proposition}
We restricted ourselves to smooth boundary data, as it will be enough to construct the harmonic measure. There is no doubt that it would be possible to obtain the same result for much less regular functions $g$.
\begin{proof}
If $u\in W(\Omega)$ is solution to $Lu=0$ and satisfies $u-\chi\in W_0(\Omega)$, then $u-\chi$ is the solution in $W_0(\Omega)$ of $Lv=-\div(-wA\grad\chi)$ given by Theorem \ref{th:solution_W}. Conversely, $v+\chi$ satisfies $L(v+\chi)=0$ and $(v+\chi)-\chi\in W_0(\Omega)$. 
\end{proof}

This result will allow us to find solutions of \eqref{eq:dir_pb} for boundary data in $\CC_c(\dOmega)$. It is well known that such solutions can be represented through the elliptic measure. Let us give a definition of this object, as well as some fundamental properties.

\begin{theorem}\label{th:elliptic_measure^}
For every $x\in\Omega$, there exists a positive Borel probability measure $\omega^x$ on $\dOmega$ such that for any $g\in\CC_c(\dOmega)$, if we write 
\begin{equation*}
u(x):=\int_\dOmega g\domega^x \for x\in\Omega,
\end{equation*}
then $u$ can be extended into a continuous bounded function on $\bOmega$ satisfying $u=g$ on $\dOmega$, and is a weak solution in $W\loc(\Omega)$ of $Lu=0$. Furthermore: 
\begin{enumerate}[(i)]
\item if $g$ is in $\CC^\infty_c(\R^n)$, then $u$ is the solution given by Proposition \ref{prop:solution_W_chi};
\item if $B$ is a ball centered on $\dOmega$ such that $g=0$ on $2B$, then $u\in W_0(B\cap\Omega)$.
\end{enumerate}
\end{theorem}
\begin{proof}
These results are included in the statement of Lemmas 12.13 and 12.15 in \cite{david_elliptic_2023}; see also Lemmas 9.4 and 9.6 in \cite{DFM21}. We will not give a full proof, and will only explain how to adapt the proof given in these references to our setting. First, for any $g\in\CC^\infty_c(\R^n)$, Lemma 11.32 in \cite{david_elliptic_2023} implies that the solution to \eqref{eq:dir_pb} given by Proposition \ref{prop:solution_W_chi} can be extended into a continuous function in $\bOmega$: let us write it $Ug$. This defines a  continuous linear operator 
\begin{equation*}
U: \CC^\infty_c(\R^n)\longrightarrow \CC_b(\bOmega),
\end{equation*} 
where $\CC_b(\bOmega)$ is the space of bounded continuous functions on $\bOmega$. Indeed, for any $g\in\CC^\infty_c(\R^n)$, the maximum principle (Lemma 12.8 in \cite{david_elliptic_2023}) tell us that
%\begin{equation}\label{eq:elliptic_measure:sup_principle}
%\sup_{\Omega} Ug \leq \sup_{\dOmega} g
%\end{equation}
%and
%\begin{equation}\label{eq:elliptic_measure:inf_principle}
%\inf_{\Omega} Ug \leq \inf_{\dOmega} g;
%\end{equation}
%in particular,
\begin{equation*}
\sup_{\Omega} |Ug| \leq \sup_{\dOmega} |g|.
\end{equation*}
Note that since $Ug-g\in W_0(\Omega)$, \eqref{eq:cv_W0} and the continuity of $Ug$ implies that
\begin{equation}\label{eq:elliptic_measure:boundary_cd}
Ug=g \texteq{ on } \dOmega.
\end{equation}
Moreover, $\CC^\infty_c(\R^n)$ (or, more precisely, the restrictions of functions in $\CC^\infty_c(\R^n)$ to $\dOmega$) is dense in $\CC_c(\dOmega)$ for the norm $\|\cdot\|_\infty$. Indeed, we know that any function in $\CC_c(\dOmega)$ can be extended into a function of $\CC_c(\R^n)$ (see Proposition VI.2.2 in \cite{Ste70} for instance), and that $\CC^\infty_c(\R^n)$ is dense in $\CC_c(\R^n)$ for $\|\cdot\|_\infty$. This allow us to uniquely extend $U$ into a continuous linear operator $U:\CC_c(\dOmega)\to\CC_b(\bOmega)$. Then, for any $x\in\Omega$, the Riesz representation theorem give us the existence of a bounded Borel measure on $\dOmega$ such that
\begin{equation*}
Ug(x)=\int_\dOmega g\domega^x.
\end{equation*}
Now that the elliptic measure is constructed, its properties can be proven in the exact same way as in Lemmas 12.13 and 12.15 in \cite{david_elliptic_2023}.

%The positivity of $\omega^x$ and the fact that $\omega^x(\dOmega)\leq 1$ is a consequence of \eqref{eq:elliptic_measure:inf_principle} and \eqref{eq:elliptic_measure:sup_principle}, respectively. The proof that $\omega^x\geq 1$ uses the H\"older continuity at the boundary: we refer to the proof of Lemma 9.6 in \todo{ref higher codim} for further details.
%\newline
%
%Finally, let us fix $g\in\CC_c(\dOmega)$, and explain briefly why $u:=Ug$ satisfy the properties of the statement. the property \eqref{eq:elliptic_measure:boundary_cd} is preserved by extension. The fact that $u$ is a solution in $W\loc(\Omega)$ to $Lu=0$ is true by definition of $U$ if $g\in\CC^\infty_c(\R^n)$, and extends to all $g\in\CC_c(\Omega)$ thanks to the interior Caccioppoli inequality \eqref{eq:int_caccio}. Property (i) is satisfied by definition of $U$. As for property (ii), if $g\in\CC^\infty_c(\R ^n)$ it is a direct consequence of the fact that $u-g\in W_0(\Omega)$, and can be extended to any $g\in\CC_c(\Omega)$ using a boundary Caccioppoli inequality (see Lemma 11.15 in \cite{david_elliptic_2023}) to control the gradient.

\end{proof}

For any Borel set $E\incl\dOmega$, the function $x\mapsto\omega^x(E)$ satisfies the following properties:

\begin{proposition}[Lemma 12.19 in \cite{david_elliptic_2023}]\label{prop:omega_sol}

Let $E\incl\dOmega$ be a Borel set, and define $u_E$ on $\Omega$ by $u_E(x)=\omega^x(E)$. Then $u_E$ belongs to $W\loc(\Omega)$ and is a solution of $Lu_E=0$ in $\Omega$. Furthermore, if $B$ is a ball centered on $\dOmega$ such that $2B\cap E=\emptyset$, then $u_E\in W_0(B\cap\Omega)$.
\end{proposition}

We then present a classical result of non-degeneracy of the elliptic measure. For completeness, we will give a proof that works in our setting.
\begin{proposition}
There exists $c>0$ such that for any ball $B$ centered on $\dOmega$,
\begin{equation}\label{eq:ndeg_harmo}
\omega^x(4B\cap\dOmega)\geq c \for x\in B.
\end{equation}
\end{proposition}

\begin{proof}
Take $B=B(\xi,r)$ centered on $\dOmega$. Let us define $u(x):=1-\omega^x(4B)=\omega^x(\dOmega\priv 4B)$ for $x\in\Omega$. According to Proposition \ref{prop:omega_sol}, $u$ is a solution in $W_0(2B\cap\Omega)$ of $Lu=0$. Then, for any $x\in B$:
\begin{equation*}
u(x)
=\limstack{y\to\xi}{y\in \gamma(\xi)} \abs{u(x)-\fint_{B_y/2} u\dm}
\leq  \limstack{y\to\xi}{y\in \gamma(\xi)} \fint_{B_y/2} |u(x)-u(z)|\dm(z)
\leq \osc{B\cap\Omega} u;
\end{equation*}
we can then use \eqref{eq:osc_estimate} together with the fact that $0\leq u\leq 1$ to obtain
\begin{equation*}
u(x)
\leq (1-\epsilon)\osc{2B\cap\Omega}u
\leq 1-\epsilon,
\end{equation*}
i.e. $\omega^x(4B)\geq \epsilon$, which is the desired result.
\end{proof}

%From the non-degeneracy of the elliptic measure, we can deduce the following result (without using connectedness) thanks to a classical argument involving the maximum principle and the upper estimate \eqref{eq:estimee_green} on the Green function; see lemma 15.28 in \cite{david_elliptic_2023}.

\begin{proposition}\label{prop:comp_green_harmo}
For any $0<\kappa<1$ and any $0<\lambda\leq 1/2$, any ball $B=B(\xi,r)$ centered on the boundary and any $y\in B$ such that $\delta(y)>\kappa r$, 
\begin{equation}\label{eq:comp_green_harmo}
G(x,y)\leq C_{\kappa,\lambda}\frac{r^2}{m(B\cap\Omega)}\omega^x(8B) \for x\in\Omega\priv \lambda B_y,
\end{equation}
where $C_{\kappa,\lambda}>0$ depends only on $n, C_{crk}, C_\mu, C_m, C_\rho, C_{pc}, C_A, \lambda$ and $\kappa$.
\end{proposition}

\begin{proof}
We will closely follow the first part of the proof of Lemma 15.28 in \cite{david_elliptic_2023}. We will use the following version of the maximum principle (Lemma 14.33 in \cite{david_elliptic_2023}, proven in \cite{DFM21}, Lemma 11.3):
\begin{lemma}\label{lemma:max_principle}
Let $D$ be an open set in $\R^n$ and $F$ a closed set in $R^n$ such that $F\incl D$ and $\texteq{dist}(F,\R^n\priv D)>0$. Let $u\in W(D\cap\Omega)$ be a solution of $Lu=0$ in $D\cap\Omega$ such that
\begin{enumerate}[(i)]
\item there exists a function $\chi\in\CC^\infty_c(\R^n)$ such that $\chi\geq 0$ $\mu$-a.e. and $u-\chi\in W_0(D\cap\Omega)$,
\item $u\geq 0$ a.e. in $D\priv F\cap\Omega$.
\end{enumerate} 
Then, $u\geq 0$ a.e. in $D\cap\Omega$.
\end{lemma}
Assumption $(i)$ is a way to impose that $u\geq 0$ on $D\cap\dOmega$, while assumption $(ii)$ imposes that $u\geq 0$ in a neighborhood of $\partial D\cap\Omega$.
\newline

Let us begin the proof of Proposition \ref{prop:comp_green_harmo}. take $B=B(\xi,r)$ centered on the boundary and $y\in B$ such that $\delta(y)>\kappa r$. First, note that
\begin{equation}\label{eq:comp_green_harmo:majo_green}
G(x,y)
\siml \frac{|x-y|^2}{m(B(y,|x-y|)\cap\Omega)}
\siml \frac{r^2}{m(B\cap\Omega)} \for x\in \lambda B_y\priv \frac{\lambda}{2} B_y,
\end{equation}
using \ref{eq:majo_green} (on $G^*$) for the first estimate and the doubling properties of $m$ for the second one. Next, we take $\chi\in \CC^\infty_c(\R^n)$ such that $\UN_{6B}\leq\chi\leq \UN_{8B}$. Let us write $u(x)=\int_\dOmega \chi\domega^x$ for $x\in\Omega$. By non-degeneracy of the harmonic measure \eqref{eq:ndeg_harmo}, we have
\begin{equation*}
u(x)\geq \omega^x(6B)\gtrsim 1 \for x\in \frac{3}{2}B.
\end{equation*}
Together with \eqref{eq:comp_green_harmo:majo_green}, this implies the existence of a constant $C_{\kappa,\lambda}>0$ such that
\begin{equation}\label{eq:comp_green_harmo:mino_v}
v(x):=C_{\kappa,\lambda} \frac{r^2}{m(B\cap\Omega)}u(x)-G(x,y) \geq 0 \for x\in \lambda B_y\priv \frac{\lambda}{2} B_y.
\end{equation}
By Theorem \ref{th:Green} and Theorem \ref{th:elliptic_measure^}, $v$ is a solution in $W(\Omega\priv\frac{\lambda}{2}\bar B_y)$ of $Lu=0$, and $v-\Ext\chi_{|\dOmega} \in W_0(\Omega\priv\frac{\lambda}{2}\bar B_y)$, so we can use Lemma \ref{lemma:max_principle} with $D=\R^n\priv\frac{\lambda}{2}\bar B_y$ and $F=\R^n\priv \lambda B_y$, extending \eqref{eq:comp_green_harmo:mino_v} to all $x\in\Omega\priv\frac{\lambda}{2}\bar B_y$. Since $u(x)\leq \omega^x(8B)$ for $x\in\Omega$, we obtain the desired result.
\end{proof}

\subsection{Tent spaces}

Take $1\leq p\leq\infty$. From now on, $\|\cdot\|_p$ will always refer to $\|\cdot\|_{L^p(\dOmega,\mu)}$. The functionals $\NNt$ and $\AAt$ introduced in \eqref{eq:def_Nt} and \eqref{eq:def_At} can be used to define the spaces
\begin{equation}\label{eq:def_modified_spaces}
\Nt^p(\Omega)=\br{u\in L^2\loc(\Omega,m)\tq \|\NNt u\|_p<\infty}
\texteq{ and } \At^p(\Omega)=\br{f\in L^2\loc(\Omega,m)\tq\|\AAt f\|_p<\infty},
\end{equation}
equipped with the obvious norm. These are a modified version of the tent spaces introduced in the case of $\Omega=\R^n_+$ by Coifman, Meyer and Stein in \cite{coifman_new_1985}. 
\newline

We have equivalence under changes of parameters for $\NN$, $\NNt$ and $\AAt$ (lemmas \ref{lemmaA:eqvlce_N} and \ref{lemmaA:eqvlce_A} in the appendix).

\begin{lemma}\label{lemma:eqvlce_NtAt}
Take $1\leq p \leq\infty$. The $L^p$ norms of $\NN^{(\alpha)}$, $\NNt^{(\alpha,\lambda)}$ and $\AAt^{\alpha,\lambda)}$ are equivalent under change of $\alpha$ and $\lambda$, i.e. for any $\alpha,\alpha'>0$, any $0<\lambda,\lambda'\leq 1/2$ and any $u\in L^2\loc(\Omega,m)$:
\begin{equation*}
\|\NN^{(\alpha)} u\|_p \leq C_{\alpha,\alpha'} \|\NN^{(\alpha')} u\|_p,
\end{equation*}
where $C_{\alpha,\alpha'}$  depends only on $C_\mu, \alpha, \alpha'$ \emph{(and only on $C_\mu$ and $\alpha/\alpha'$ if we further impose $\alpha'>\alpha\geq 1$)}, and
\begin{equation*}
\|\NNt^{(\alpha,\lambda)} u\|_p \leq C_{\alpha,\alpha',\lambda/\lambda'} \|\NNt^{(\alpha',\lambda')} u\|_p,
\end{equation*}
\begin{equation*}
\|\AAt^{(\alpha,\lambda)} u\|_p \leq C_{\alpha,\alpha',\lambda/\lambda'} \|\AAt^{(\alpha',\lambda')} u\|_p,
\end{equation*}
where $C_{\alpha,\alpha',\lambda/\lambda'}$ depends only on $C_{crk}, C_m, C_\mu, \alpha,\alpha'$, and $\lambda/\lambda'$.
\end{lemma}

\begin{lemma}\label{lemma:tent_space_complete}
Take $1<p<\infty$. $\Nt^p(\Omega)$ and $\At^p(\Omega)$ are Banach spaces.
\end{lemma}
\begin{proof}
See Theorem \ref{thA:Banach} in the Appendix.
\end{proof}

Then, we present a theorem encompassing all the results on duality between $\NNt^p(\Omega)$ and $\AAt^p(\Omega)$ which are established in the appendix (Theorems \ref{thA:dual_intNA}, \ref{thA:dual_A'N} and \ref{thA:dual_Asup}).

\begin{theorem}
Take $1<p<\infty$. The dual space of $\At^p(\Omega)$ is homeomorphic to $\Nt^{p'}(\Omega)$. Furthermore, we have the following estimates:

\begin{equation}\label{eq:dual_prodNtAt}
\int_\Omega uf \frac{\dm}{\rho\delta}
\leq C\int_\dOmega (\NNt u)(\AAt f)\dmu
\leq C\|\NNt u\|_p\|\AAt f\|_{p'} \for u\in\Nt^p(\Omega), f\in\At^p(\Omega),
\end{equation}

\begin{equation}\label{eq:dual_Ntsup}
\|\NNt u\|_{p'} \leq C \sup_{\|\AAt f\|_p\leq 1} \int_\Omega uf\frac{\dm}{\rho\delta}\for u\in\Nt^p(\Omega),
\end{equation}

\begin{equation}\label{eq:dual_Atsup}
\|\AAt f\|_{p'} \leq C \sup_{\|\NNt u\|_p\leq 1} \int_\Omega uf\frac{\dm}{\rho\delta}\for f\in\At^p(\Omega),
\end{equation}
where $C>0$ depends only on $C_{crk}, C_m$ and $C_\mu$.
\end{theorem}

Finally, we will also make use of the following density result:
\begin{theorem}\label{th:densite_At}
The subspace of Lipschitz continuous, compactly supported functions on $\Omega$ is dense in $\At^p(\Omega)$.
\end{theorem}
\begin{proof}
It is easy to see that $L^2_c(\Omega,m)$ is dense in $\At^p(\Omega)$, using Lebesgue's dominated convergence theorem, so it is enough to prove that functions of $L^2_c(\Omega,m)$ can be approximated in $\At^p(\Omega)$ by Lipschitz continuous, compactly supported functions. But this is proven in \cite{mourgoglou_solvability_2023} (lemma 2.6), in a way that can be generalized to our geometric setting with no difficulties.
\end{proof}

%%%%%%%%%%%%%%%%%%%%%%%%%%%%%%%%%%%%%%%%%%%%%%%%%%%%%%%%%%%%%%%%%%%%%
%D_p
%%%%%%%%%%%%%%%%%%%%%%%%%%%%%%%%%%%%%%%%%%%%%%%%%%%%%%%%%%%%%%%%%%%%%

\section{Characterization of \texorpdfstring{$(D_p)$}{(Dp)}}\label{s:Dp}

\subsection{Characterization using the elliptic measure}
It is well known that the solvability of $(D_p)$ for some $1<p<\infty$ is equivalent to the fact that the elliptic measure $\omega$ satisfies a (weak) $L^\pp$-$L^1$ reverse Hölder inequality with respect to the boundary measure (see for example \cite{DahlbergBjörn}). Our aim in this subsection is to prove this result in our geometric setting. First, let us give a precise statement.

\begin{theorem}\label{th:DpRHp'}
If $(D_p)$ is solvable for some $1<p<\infty$, then $\omega$ is \emph{absolutely continuous} with respect to $\mu$ (noted $\omega\ll\mu$), i.e.
\begin{equation*}
\mu(E)=0 \Longrightarrow \omega^x(E)=0 \texteq{ for any Borel set } E\incl \dOmega \emph{ and any } x\in\Omega.
\end{equation*}
This implies that for any $x\in\Omega$ there exists a positive measurable function on $\dOmega$, the \emph{Radon-Nikodym derivative} of $\omega^x$ with respect to $\mu$, noted $\frac{d\omega^x}{d\mu}$, such that $d\omega^x=\frac{d\omega^x}{d\mu}d\mu$. Furthermore, for any $0<\kappa<1$, any ball $B=B(\xi,r)$ centered on $\dOmega$ and any $x\in B$ satisfying $\delta(x)>\kappa r$,
\begin{equation}\label{eq:RHp'}
\pfrac[p']{\fint_{B\cap \partial \Omega} \abs{\frac{d\omega^x}{d\mu}}^\pp \dmu }
\leq \frac{C_\kappa}{\mu(B\cap \partial \Omega)},
\end{equation} 
Where $C_\kappa>0$ depends only on $C_\mu, \kappa$ and the solvability constant in $(D_p)$.
Conversely, there exists $0<\kappa_0<1$ such that if $\omega\ll\mu$ and \eqref{eq:RHp'} holds for $\kappa_0$ and some $1<p<\infty$, then $(D_p)$ is solvable. 
\end{theorem}

As stated, this is a characterization of the solvability of $(D_p)$ by a weak reverse Hölder inequality on the elliptic measure: indeed, thanks to the non-degeneracy of $\omega$, \eqref{eq:RHp'} is equivalent to
\begin{equation}\label{eq:true_RHp'}
\pfrac[p']{\fint_{B\cap \partial \Omega} \abs{\frac{d\omega^x}{d\mu}}^\pp \dmu }
\leq C_\kappa\frac{\omega^x(8B \cap \partial \Omega)}{\mu(8B \cap \partial \Omega)}
=C_\kappa\fint_{8B \cap \partial \Omega} \frac{d\omega^x}{d\mu}\dmu.
\end{equation}
It is well known that if the reverse Hölder inequality \eqref{eq:true_RHp'} holds for $p'$, then it also holds for some $q'>p'$ close to $p'$ (see \cite{giaquinta_multiple_83}, Chapter 5, Proposition 1.1). In particular, we have that if $(D_p)$ is solvable, then there exists some $q<p$ such that $(D_q)$ is solvable.
\newline

Before proving the theorem, let us give a useful lemma.
\begin{lemma}\label{lemma:uNu}
For any ball $B$ centered on $\dOmega$ and any $u\in W_0(2B\cap\Omega)$ solution of $Lu=0$ in $2B\cap\Omega$, we have
\begin{equation}\label{eq:uNu}
\sup_{B\cap\Omega} |u|
\leq C\fint_{8B\cap \partial \Omega} \NN (u\UN_{B}) \dmu,
\end{equation}
where $C>0$ depends only on $n, C_{crk}, C_\mu, C_m, C_\rho, C_{pc}$ and $C_A$.
\end{lemma}

\begin{proof}
Let us make two remarks before beginning the proof.
\begin{itemize}
\item since we do not have (quantitative) connectedness of $\dOmega$, we cannot prove a change of pole theorem such as Lemma 15.14 of \cite{david_elliptic_2023}. Nevertheless, this weaker statement holds, and can be proven in a similar fashion.
\item This result is far easier to prove in the case of, for instance, $(n-1)$-Ahlfors regular boundaries. Indeed, in this case the estimate
\begin{equation*}
\fint_{B \cap \Omega} u\dm
\siml \fint_{8B \cap \partial \Omega} \NN (u\UN_B)\dmu
\end{equation*}
holds for \emph{any} measurable $u$, using $\NN$-$\AA$ duality; then using some Moser estimate on the boundary we obtain the result. If we had imposed a lower bound in (H4), say
\begin{equation*}
\rho(\Lambda B)\geq C_4\inv \Lambda^{\epsilon-1} \rho(B),
\end{equation*}
we could proceed in the same way; but we do not assume such bound, thus more work is required.
\end{itemize}

Let us start the proof. First, recall that the Hölder regularity of $u$ at the boundary \eqref{eq:boundary_holder} gives us $0<\epsilon<1/2$ such that
\begin{equation}\label{eq:uNu_eps_Holder}
\sup_{\epsilon B'\cap\Omega} |u|
\leq \frac{1}{2}\sup_{B'\cap\Omega} |u|,
\end{equation}
for any ball $B'\incl 2B$ centered on $\dOmega$.
\newline

Now, we will prove that there exists $M>0$ such that for any integer $k$ and any $x\in 2B$,
\begin{equation}\label{eq:uNu_prelim}
\delta(x)\geq \epsilon^kr
\Longrightarrow |u(x)|\leq M^k \fint_{8B \cap \partial \Omega} \NN (u\UN_B) \dmu.
\end{equation}
In order to do so, we will use the fact that for any $v$ measurable and any $\alpha>0$,
\begin{equation}
\int_\dOmega \NN^{(2\alpha)} v \dmu
\leq C \int_\dOmega \NN^{(\alpha)} v \dmu,
\end{equation}
with a constant $C>0$ independent from $\alpha$ (see Lemma \ref{lemma:eqvlce_NtAt}). Fix an integer $a$ such that $2^a\geq 4/\epsilon$, and consider $x\in 2B$ satisfying $\delta(x)\geq \epsilon^kr$ for some integer $k$.  Then $x\in \gamma^{(2^{ak})}(\zeta)$ for any $\zeta\in B$, implying that 
\begin{equation*}
|u(x)|\leq \NN^{(2^{ak})} (u\UN_B) (\zeta) \for \zeta\in B,
\end{equation*}
which in turn gives us
\begin{align*}
|u(x)|
\leq \fint_{B \cap \partial \Omega} \NN^{(2^{ak})} (u\UN_B) \dmu
\leq \frac{1}{\mu(B)} \int_\dOmega \NN^{(2^{ak})} (u\UN_B) \dmu
&\leq \frac{C^{ak}}{\mu(B)} \int_\dOmega \NN (u\UN_B) \dmu\\
&\siml C^{ak}\fint_{8B \cap \partial \Omega} \NN (u\UN_B) \dmu,
\end{align*}
by the doubling property of $\mu$ (assumption (H2)). The bound \eqref{eq:uNu_prelim} follows.
\newline

With this result proven, we can conclude in exactly the same manner as in Lemma 4.4 of \cite{JK82} or Lemma 15.14 of \cite{david_elliptic_2023}. Nonetheless, we will reproduce it for completeness.
\newline

Consider an integer $i$ such that $2^i>M$, and set $M'=M^{i+3}$. We can rescale $u$ and assume that $\fint_{8B} \NN (u\UN_B)\dmu=1$. We will prove by contradiction that 
\begin{equation*}
\sup_{B\cap\Omega} |u|
\leq M'; 
\end{equation*}
thus, let us assume that there exists $x_1\in B\cap\Omega$ such that $|u(x_1)|> M'$. We will prove by induction that for any integer $k\geq 1$,
\begin{equation}\label{eq:uNu_rec}
\texteq{there exists } x_k\in\Omega \texteq{ such that } |u(x_k)|>M^{i+2+k} \texteq{ and } x_k\in (\frac{3}{2}-2^{-k})B
\end{equation}
The first step is given by our assumption, so let us prove the induction. Take an integer $k\geq 1$ and assume \eqref{eq:uNu_rec} true. From the contraposition of \eqref{eq:uNu_prelim}, we deduce that $\delta(x_k)<\epsilon^{i+2+k}r$. Choose $\xi_k\in\dOmega$ such that $|x_k-\xi_k|=\delta(x_k)$. Using the induction hypothesis and the fact that $\epsilon\leq 1/2$, we have
\begin{equation*}
|\xi_k-\xi|
\leq \delta(x_k) + |x_k-\xi|
\leq \p{\frac{3}{2}-2^{-k}+2^{-2-k}}r,
\end{equation*}
which implies that
\begin{equation}\label{eq:uNu_incl_boules}
B(\xi_k,\epsilon^{2+k}r)\incl (\frac{3}{2}-2^{-k-1})B;
\end{equation}
in particular $B(\xi_k,\epsilon^{2+k}r)\incl 2B$. By iterating \eqref{eq:uNu_eps_Holder}, we have that 
\begin{equation*}
\sup_{B(\xi_k,\epsilon^{2+k}r)\cap\Omega} |u|
\geq 2^i\sup_{B(\xi_k,\epsilon^{i+2+k}r)\cap\Omega} |u|,
\end{equation*}
and since $x_k\in B(\xi_k,\epsilon^{i+2+k}r)$,
\begin{equation*}
\sup_{B(\xi_k,\epsilon^{2+k}r)\cap\Omega} |u|
\geq 2^i |u(x_k)|
>2^iM^{i+2+k}
>M^{i+2+k+1}.
\end{equation*}
We can thus find $x_{k+1}\in B(\xi_k,\epsilon^{2+k}r)\cap\Omega$ such that $|u(x_{k+1})|>M^{i+2+k+1}$. Furthermore, \eqref{eq:uNu_incl_boules} implies that $x_{k+1}\in (\frac{3}{2}-2^{-k-1})B$. This complete the induction step.
\newline

To sum up, we have proven that if there exists $x_1\in B$ such that $|u(x_1)|>M'$, then there exists a sequence $(x_k)$ in $\frac{3}{2}B\cap\Omega$ such that $|u(x_k)|$ goes to infinity. In this case, we can extract a subsequence to find a point in $\frac{3}{2}\bar B\cap\bOmega\incl 2B\cap\bOmega$ where $u$ is not continuous, which contradicts the (Hölder) continuity of the solution. Hence the desired result.
\end{proof}

Now we return to the proof of the main theorem.

\begin{proof}[Proof of Theorem \ref{th:DpRHp'}]

Our proof is inspired by Theorem 9.2 in \cite{mourgoglou_regularity_2023}; see also Theorem 1.3 in \cite{cao_absolute_2022}. We begin by proving the first statement: let us take $0<\kappa<1$, a ball $B=B(\xi,r)$ centered on $\dOmega$ and $x\in B$ such that $\delta(x)>\kappa r$. We also fix $g\in \CC_c(\dOmega)$ such that $\supp g \incl B$. We will first show that

\begin{equation}\label{eq:Dp_RHp':dual}
\int_\dOmega g\domega^x \siml \frac{1}{\mu(B)^{1/p}}\Lpd{g}.
\end{equation}
Write $u(x)=\int_\dOmega g\domega^x$. We fix $\alpha=2/\kappa$, so that $B\incl (1+\alpha) B_x$; since $u(x)\leq \NN^{(\alpha)}u(\zeta)$ for $\zeta\in (1+\alpha) B_x$, we have

\begin{equation*}
u(x)
\leq \fint_{(1+\alpha)B_x} \NN^{(\alpha)} u \dmu
\siml \frac{1}{\mu(B)^{1/p}} \|\NN u\|_p
\siml \frac{1}{\mu(B)^{1/p}} \|g\|_p,
\end{equation*} 
where we used the equivalence under change of parameter (Lemma \ref{lemma:eqvlce_NtAt}) for the second estimate and the solvability of $(D_p)$ for the third one. This gives us \eqref{eq:Dp_RHp':dual}.
\newline

Next, let us prove that $\omega\ll\mu$. Fix $x\in\Omega$, $\xi\in 2B_x\cap\dOmega$ and write $B=B(\xi,2\delta(x))$. Take a Borel set $E\incl\dOmega$ such that $\mu(E)=0$; we will prove that $\omega^x(E)=0$. Since $\mu$ and $\omega^x$ are Borel measures on the separable complete metric space $\dOmega$, they are regular. This means that if we take $\epsilon>0$, there exists $K\incl E\incl U\incl \dOmega$ with $K$ compact, $U$ open in $\dOmega$ such that $\mu(U\priv K) + \omega^x(U\priv K) <\epsilon$. By Urysohn's lemma, there exists $g\in\CC_c(\dOmega)$ such that $\UN_K\leq g \leq \UN_U$. Then, using \eqref{eq:Dp_RHp':dual}:

\begin{align*}
\omega^x(E)
\leq \epsilon + \omega^x(K)
\leq \epsilon + \int_\Omega g\domega^x
\siml \epsilon + \mu(B)^{-1/p} \|g\|_p
&\siml \epsilon + \mu(B)^{-1/p} \mu(U)^{1/p}\\
&\siml \epsilon + \mu(B)^{-1/p}\epsilon^{1/p};
\end{align*}
since the constant is independent of $\epsilon$, we obtain that $\omega^x(E)=0$.
\newline

Finally, now that we know that $\omega\ll\mu$, \eqref{eq:Dp_RHp':dual} gives us 
\begin{equation*}
\pfrac[\pp]{\int_\dOmega \abs{\frac{d\omega^x}{d\mu}}^\pp\dmu}
=\sup_{\|g\|_p\leq 1}\int_\dOmega \frac{d\omega^x}{d\mu} g \dmu
=\sup_{\|g\|_p\leq 1}\int_\dOmega g\domega^x
\siml \frac{1}{\mu(B)^{1/p}},
\end{equation*}
since the supremum can be taken by density on the set of functions $g\in\CC_c(\dOmega)$ satisfying $\|g\|_p\leq 1$. This implies \eqref{eq:RHp'}.
\newline

Now, let us prove the converse statement. Take $0<\kappa_0<1$ to be fixed later, and let us assume that $\omega\ll\mu$ and that the reverse Hölder inequality \eqref{eq:RHp'} holds for $\kappa_0$ and some $1<p<\infty$. Then, Proposition 1.1 in Chapter 5 of \cite{giaquinta_multiple_83} tells us that it also holds for some $q<p$, which only depends on $p$. Fix $g\in \CC_c(\dOmega)$, and consider $u$ the solution given by Theorem \ref{th:elliptic_measure^}. We can assume that $g\geq 0$: indeed, by the triangular inequality and the decomposition of $g$ in positive and negative parts, if we show that $\|\NN u\|_p\siml \|g\|_p$ then $(D_p)$ will be solvable. Take $\xi\in\dOmega$ and $x\in\gamma(\xi)$, and write $B=B(\xi,\delta(x))$. Consider $K>1$ to be fixed later, and take $\chi\in\CC^\infty_c(\R^n)$ such that $\UN_{4KB}\leq \chi\leq \UN_{8KB}$. We will then decompose $g=\chi g + (1-\chi)g =: g_0+g_1$, and consider the corresponding solutions $u_0$ and $u_1$; we have $u=u_0+u_1$, $u_0\geq 0$, and $u_1\geq 0$. First, let us estimate $u_0(x)$: using the Hölder inequality,

\begin{align}\label{eq:DpRHp'_est_near_x}
0\leq u_0(x)
=\int_{8KB} g_0 \frac{d\omega^x}{d\mu} \dmu
&\leq \pfrac[q]{\int_{8KB} |g|^q\dmu} \pfrac[\qp]{\int_{8KB} \abs{\frac{d\omega^x}{d\mu}}^\qp\dmu}\nonumber\\ 
&\siml \pfrac[q]{\fint_{8KB} |g|^q\dmu} 
\leq \MM_{\mu,q} g(\xi),
\end{align}
if we assume that \eqref{eq:RHp'} holds for $\kappa_0=\frac{1}{16K}$, and where $\MM_{\mu,q} g$ is the centered $L^q$ Hardy-Littlewood maximal function, defined by 
\begin{equation*}
\MM_{\mu,q} g(\xi):=\sup_{r>0} \pfrac[q]{\fint_{B(\xi,r)}|g|^q\dmu} \for \xi\in\dOmega.
\end{equation*}
Then, let us estimate $u_1(x)$: if we write $\MM_\mu:=\MM_{\mu,1}$ the centered Hardy-Littlewood maximal function on $L^1\loc(\dOmega,\mu)$,

\begin{equation}\label{eq:DpRHp'_est_far_x}
0\leq u_1(x)
\siml K^{-\eta} \pfrac{\fint_{KB} u_1 \dm}
\siml K^{-\eta} \fint_{8KB} \NN (u_1\UN_{KB}) \dmu
\siml K^{-\eta} \MM_\mu \NN (u_1\UN_{KB}) (\xi),
\end{equation}
where we used the fact that $u_1\in W_0(2KB\cap\Omega)$ (Theorem \ref{th:elliptic_measure^}) to apply the Hölder regularity estimate \eqref{eq:boundary_holder} for the first estimate, and Lemma \ref{lemma:uNu} for the second. Piecing together \eqref{eq:DpRHp'_est_near_x} and \eqref{eq:DpRHp'_est_far_x}, we obtain that

\begin{equation}
\NN u(\xi)
\leq C_K\MM_{\mu,q} g(\xi) + \frac{C}{K^\eta} \MM_\mu\NN (u_1\UN_{KB})(\xi) \for \xi\in\dOmega,
\end{equation}
which implies, using the $L^p$-boundedness of $\MM_{\mu,q}$ when $p>q$:

\begin{equation}\label{eq:Dp_RHp':Nu}
\|\NN u\|_p \leq C'_K\|g\|_p + \frac{C'}{K^\eta}\|\NN (u_1\UN_{KB})\|_p.
\end{equation}
Thanks to the Hölder regularity estimate \ref{eq:boundary_holder}, we have that $u_1\UN_{KB}\in L^\infty_c(\Omega,m)$. This implies by a simple estimate that $\NN(u_1\UN_{KB})\in L^\infty_c(\dOmega,\mu)\incl L^p(\dOmega,\mu)$. Then, \eqref{eq:Dp_RHp':Nu} give us that $\|\NN u\|_p$ is finite and
\begin{equation*}
\|\NN u\|_p \leq C'_K\|g\|_p + \frac{C'}{K^\eta}\|\NN u\|_p,
\end{equation*}
since $0\leq u_1\leq u$. Thus, we can choose $K$ big enough with respect to $C'$ so that $\|\NN u\|_p\siml \|g\|_p$, which proves the result with $\kappa_0=\frac{1}{16K}$.
\end{proof}

\subsection{Characterization using the Green function}

Now, we will use Theorem \ref{th:DpRHp'} to prove another characterization of the solvability of $(D_p)$, this time using estimates on the Green function. This is a generalized version of the end of Theorem 1.22 in \cite{mourgoglou_solvability_2023}.

\begin{lemma}\label{lemma:estimee_green}
Take $1<p<\infty$. If $(D_p)$ is solvable, then for any $0<\kappa<1$, any ball $B=B(\xi,r)$ centered on $\dOmega$ and any $x\in B$ satisfying $\delta(x) >\kappa r$, we have these two estimates:
\begin{equation}\label{eq:estimee_green_grad}
\Lpd{\NNt\p{\rho\grad G(x,\cdot)\UN_{B\priv B_x/2}}}
\siml \frac{C_\kappa}{\mu(B\cap\dOmega)^{1/p}},
\end{equation}
\begin{equation}\label{eq:estimee_green}
\Lpd{\NNt\p{\frac{\rho}{\delta} G(x,\cdot)\UN_{B\priv B_x/8}}}
\siml \frac{C_\kappa}{\mu(B\cap\dOmega)^{1/p}},
\end{equation}
where $C_\kappa>0$ depends only on $n, C_{crk}, C_\mu, C_m, C_\rho, C_{pc}, C_A,\kappa$ and the solvability constant in $(D_p)$. Conversely, if either one of these estimates holds for $\kappa_1:=\kappa_0/4$ (where $\kappa_0$ is given in Theorem \ref{th:DpRHp'}), then $(D_p)$ is solvable.

\end{lemma}

\begin{proof}
\textbf{First step.} We begin by proving that if \eqref{eq:estimee_green} holds for some $0<\kappa<1/4$, then \eqref{eq:estimee_green_grad} holds for $4\kappa$. Fix $0<\kappa<1$ and write $\lambda=\kappa/32$. Take $B=B(\xi,r)$ centered on $\dOmega$ and $x\in B$ such that $\delta(x)>4\kappa r$ ; for any $\zeta\in \dOmega$ and $z\in\gamma(\zeta)$, 

\begin{align*}
\pfrac{\fint_{\lambda B_z} |\rho(y)\grad_y G(x,y)\UN_{B\priv B_x/2}(y)|^2\dm(y)}
&\siml \rho(z)\UN_{2B\priv  B_x/4}(z) \pfrac{\fint_{\lambda B_z} |\grad_yG(x,y)|^2\dm(y)}\\
&\siml \rho(z)\UN_{2B\priv  B_x/4}(z) \frac{1}{\delta(z)}\pfrac{\fint_{2\lambda B_z} |G(x,y)|^2\dm(y)}\\
&\siml \pfrac{\fint_{2\lambda B_z} \abs{\frac{\rho(y)}{\delta(y)}G(x,y)\UN_{4B\priv B_x/8}(y)}^2\dm(y)},
\end{align*}
Which directly implies the desired result. For the first estimate, we use that $\rho(y)\simeq\rho(z)$ and $\delta(y)\simeq\delta(z)$ for $y\in\lambda 2B_z$. We also use that if $y\in\lambda B_z\cap B\priv B_x/2$, then $z\in 2B\priv B_x/4$ since $r>|x-\xi|>\delta(x)>4\kappa r$. For the second estimate, we use that $x\notin\lambda B_z$ if $z\in 2B\priv B_x/4$, since $\delta(x)>4\kappa r$; thanks to Theorem \ref{th:Green}, this allow us to use the Caccioppoli inequality \eqref{eq:int_caccio}. The third estimate uses computations analogous to the ones needed for the first estimate.
\newline

\textbf{Second step}. Now, let us assume that $(D_p)$ is solvable, and prove \eqref{eq:estimee_green}. By the result of the first step, this will also imply \eqref{eq:estimee_green_grad}. Fix $0<\kappa<1$ and write $\lambda=\kappa/32$. Take $B=B(\xi,r)$ centered on $\dOmega$, $x\in B$ such that $\delta(x)>\kappa r$, $\zeta\in\dOmega$ and $z\in\gamma(\zeta)$.
\newline

First, we can assume that $z\in 2B\priv B_x/16$, since the integral below is null otherwise. Using the Harnack inequality \eqref{eq:harnack},
\begin{equation*}
\pfrac{\fint_{\lambda B_z} \abs{\frac{\rho(y)}{\delta(y)} G(x,y) \UN_{B\priv B_x/8}(y)}^2\dm(y)}
\siml \frac{\rho(z)}{\delta(z)}G(x,z).
\end{equation*}
Then, since $x\notin \lambda B_z$, we can use \eqref{eq:comp_green_harmo} with the ball $B_\zeta=B(\zeta,16\delta(z))$, giving us
\begin{equation*}
\frac{\rho(z)}{\delta(z)} G(x,z)
\siml\frac{\rho(z)\delta(z)}{m(B_\zeta)}\omega^x(B_\zeta)
\siml \frac{\omega^x(B_\zeta)}{\mu(B_\zeta)}
\siml \MM_\mu \omega^x_{|16B}(\zeta),
\end{equation*}
where in the last estimate we used the fact that $B_\zeta\incl 16B$. Finally, since $(D_p)$ is solvable:
\begin{equation*}
\|\NNt \left(\frac{\rho}{\delta}G(x,\cdot)\UN_{B\setminus B_x/8}\right)\|_\pp
\siml \Npp{\MM_\mu\omega^x_{|16B}}
\siml \pfrac[\pp]{\int_{16B} \abs{\frac{d\omega^x}{d\mu}}^\pp\dmu}
\siml \frac{1}{\mu(B)^{1/p}},
\end{equation*}
which is the desired estimate.
\newline

\textbf{Third step.} Finally, we prove the converse statement. Thanks to the first step, it is enough to prove that \eqref{eq:estimee_green_grad} with $\kappa=\kappa_0$ implies $(D_p)$ is solvable. Take $B=B(\xi,r)$ centered on $\dOmega$ and $x\in B$ satisfying $\delta(x)>\kappa_0 r$. Thanks to Theorem \ref{th:DpRHp'}, it is enough to prove that
\begin{equation}\label{eq:RHp'_fct_max}
\|\MM_\mu^x \omega^x\|_{L^{p'}(B,\mu)} \siml \frac{1}{\mu(B)^{1/p}},
\end{equation}
where $\MM_\mu^x\omega^x$ is the restricted maximal function of $\omega^x$ defined by
\begin{equation*}
\MM_\mu^x\omega^x(\zeta) = \sup_{0<s<\delta(x)/4} \frac{\omega^x(B(\zeta,s))}{\mu(B(\zeta,s))} \for \zeta\in\dOmega.
\end{equation*}
Indeed, if $\MM_\mu^x\omega^x$ is finite almost everywhere, then $\omega^x\ll\mu$ (see \cite{mattila_geometry_95}, Theorem 2.12), and \eqref{eq:RHp'_fct_max} implies \eqref{eq:RHp'}, since $\frac{d\omega^x}{d\mu}\leq \MM_\mu^x\p{\frac{d\omega^x}{d\mu}}=\MM_\mu^x\omega^x$ by Lebesgue's differentiation theorem.
\newline

Thus, we consider $\zeta\in B\cap\dOmega$, $0<s<\delta(x)/4$ and write $B_\zeta=B(\zeta,s)$. Then, we can take $\chi\in\CC^\infty_c(\R^n)$ such that $\chi=1$ in $B_\zeta$, $\chi=0$ outside of $2B_\zeta$, and $\|\grad\chi\|_\infty\siml\frac{1}{s}$. We first have:
\begin{equation*}
\omega^x(B_\zeta)
\leq \int_{\dOmega} \chi \domega^x =: u(x),
\end{equation*}
and if we set 
\begin{equation*}
v(x):=\int_\Omega A(y)^T \grad_y G(x,y)\cdot\grad\chi(y)\dm(y)
=\int_\Omega \grad_y G(x,y)\cdot A(y)\grad\chi(y)\dm(y),
\end{equation*}
we have that $v=\chi-u$. Indeed, by theorem \ref{th:Green_repr}, $v$ is in $W_0(\Omega)$ and is solution of $Lv=-\div(wA\grad\chi)$. On the other hand, by Theorem \ref{th:elliptic_measure^} and Proposition \ref{prop:solution_W_chi}, $\chi-u$ is also in $W_0(\Omega)$ and solution of $Lu=-\div(wA\grad\chi)$.  By uniqueness in Theorem \ref{th:solution_W}, we have the equality.
\newline

In particular, since $x\notin 4B_\zeta$, $v(x)=-u(x)$; hence
\begin{align*}
\omega^x(B_\zeta)
&\leq -\int_\Omega A(y)^T \grad_y G(x,y)\cdot\grad\chi(y)\dm(y)\\
&\siml \frac{1}{s}\int_{2B_\zeta} |\grad_y G(x,y)|\dm(y)\\
&\siml \frac{1}{s}\int_{2B_\zeta} \rho(y)|\grad_y G(x,y)|\UN_{2 B\priv B_x/2}(y)\delta(y)\frac{\dm(y)}{\rho(y)\delta(y)}\\
&\siml \int_{16B_\zeta} \NNt \p{\rho\grad_yG(x,\cdot)\UN_{2B\priv B_x/2}} \dmu,
\end{align*}
where for the last estimate we used duality from \eqref{eq:dual_prodNtAt} and the fact that $\AAt(\delta\UN_{2B})\siml s$ on $\dOmega$. This yields:
\begin{equation*}
\MM^x_\mu \omega^x(\zeta)
\siml \MM_\mu\p{\NNt \left(\rho\grad_yG(x,\cdot)\UN_{2B\setminus B_x/2}\right)}(\zeta) \for \zeta\in B \cap \dOmega;
\end{equation*}
passing to the $L^\pp$ norm, we can finally use \eqref{eq:estimee_green_grad} to obtain
\begin{equation*}
\|\MM^x_\mu \omega^x\|_{L^{p'}(\Lambda B,\mu)}
\siml \|\NNt \left(\rho\grad_yG(x,\cdot)\UN_{2B\setminus B_x/2}\right) \|_\pp
\siml \unsur{\mu(B)^{1/p}},
\end{equation*}
which is the desired estimate \eqref{eq:RHp'_fct_max}.

\end{proof}

%%%%%%%%%%%%%%%%%%%%%%%%%%%%%%%%%%%%%%%%%%%%%%%%%%%%%%%%%%%%%%%%%%%%%
%PD_p
%%%%%%%%%%%%%%%%%%%%%%%%%%%%%%%%%%%%%%%%%%%%%%%%%%%%%%%%%%%%%%%%%%%%%

\section{Solvability of the Poisson-Dirichlet problem}\label{s:PDp}

We will now turn to our main subject and study the solvability of $(PD_p)$. Our main goal here is to give a proof of Theorem \ref{th:eqvlce_DPD}. As a corollary, if the equivalent conditions of this theorem are satisfied, we will obtain the existence of solutions with general boundary data $g\in L^p(\dOmega,\dmu)$ and general interior data $f$ and $F$ belonging to suitable tent spaces.

\subsection{Proof of Theorem \ref{th:eqvlce_DPD}}
We now have enough elements to tackle the proof of the main theorem. Notice that the implications $(ii)\Rightarrow(iii)$, $(ii)\Rightarrow(iv)$, and $(v)\Rightarrow(vi)$ are automatic. We will prove $(i)\Leftrightarrow(iii)$, $(i)\Leftrightarrow(iv)$ (giving us $(i)\Leftrightarrow(ii)$ by linearity), $(vi)\Rightarrow(iv)$ and $(iv)\Rightarrow(v)$.
\newline

\textbf{Proof of $(i) \Rightarrow (iii)$.} The principle of this proof is similar to the one of the converse statement in Theorem \ref{th:DpRHp'} : we decompose our solution to estimate it near the pole of the Green function, then near the vertex of a cone on the boundary, and finally far from this vertex.
\newline

Fix $f\in L^\infty_c(\Omega,m)$, and consider the solution $u$ of
\begin{equation}
\begin{cases}
Lu=wf &\texteq{in } \Omega\\
u=0 &\texteq{on } \dOmega
\end{cases}
\end{equation}
given by Theorem \ref{th:solution_W}. We can assume that $f\geq 0$: indeed, by the triangular inequality and the decomposition of $f$ in positive and negative parts, if we show that $\|\NNt u\|_p\siml \|\AAt(\delta^2f)\|_p$ then $(PD_p)$ will be solvable. We will estimate $\NNt^{(1,1/16)} u (\xi) = \sup_{x\in\gamma(\xi)} \pfrac{\fint_{B_x/16} |u|^2\dm}$ for $\xi\in\dOmega$.
\newline

Take $K\geq 2$ to be fixed later. Fix $\xi\in\dOmega$ and $x\in\gamma(\xi)$, and write $B=B(\xi,4\delta(x))$. Next, decompose $f$ into $f=f\UN_{B_x/8} + f\UN_{KB\priv B_x/8} + f\UN_{\Omega\priv KB}=:f_0+f_1+f_2$, and write $u_0,u_1,u_2$ the corresponding solutions: $u=u_0+u_1+u_2$. We will estimate each term independently.
\newline

First, let us estimate $\pfrac{\fint_{B_x/8} |u_0|^2\dm}$. Remark that
\begin{align*}
\int_{B_x/8} |u_0|^2\dm
\leq \int_{B} |u_0|^2\dm
&\siml\delta(x)^2\int_\Omega|\grad u_0|^2 \dm\\
&\siml \delta(x)^2\int_\Omega A\grad u_0\cdot\grad u_0\dm\\
&\siml\delta(x)^2\int_{B_x/8} fu_0\dm\\
&\siml \pfrac{\int_{B_x/8} |\delta^2f|^2 \dm} \pfrac{\int_{B_x/8} |u_0|^2 \dm},
\end{align*}
using the boundary Poincaré inequality \eqref{eq:boundary_poincare}; hence

\begin{align*}
\pfrac{\fint_{B_x/16} |u_0|^2\dm}
\siml \pfrac{\fint_{B_x/8} |u_0|^2\dm}
&\siml \pfrac{\fint_{B_x/8} |\delta^2 f|^2\dm}\\
&\siml \int_{B_x/8} \pfrac{\fint_{B_x/8}|\delta^2 f|^2\dm} \frac{dm(y)}{m(B_y)}\\
&\siml \int_{B_x/8} \pfrac{\fint_{B_y/2}|\delta^2 f|^2\dm} \frac{dm(y)}{m(B_y)}\\
&\siml \AAt^{(2,1/2)} (\delta^2 f)(\xi).
\end{align*}

Then we estimate $u_1$ on $B_x/16$. Take $q<p$ such that $(D_q)$ is solvable. For any $z\in B_x/16$:
\begin{align*}
|u_1(z)|
=\abs{\int_{KB\priv B_x/8} G(z,y) f(y) \dm(y)}
&\siml \int_{KB\priv B_z/32} \abs{\frac{\rho(y)}{\delta(y)} G(z,y)}|\delta(y)^2f(y)|\frac{\dm(y)}{\rho(y)\delta(y)}\\
&\siml \Nqp{\NNt\p{\frac{\rho}{\delta} G(z,\cdot)\UN_{KB\priv B_z/32}}} \Nq{\AAt(\delta^2 f\UN_{KB})}\\
&\siml \frac{1}{\mu(KB)^{1/q}}\pfrac[q]{\int_{8KB}\AAt(\delta^2 f)^q\dmu}\\
&\siml \pfrac[q]{\fint_{8KB} \AAt(\delta^2 f)^q\dmu}
\leq \MM_{\mu,q}\AAt(\delta^2 f)(\xi),
\end{align*}
where we used Theorem \ref{th:Green_repr} for the first equality, duality from \eqref{eq:dual_prodNtAt} for the second estimate, and Lemma \ref{lemma:estimee_green} for the third one. Here, the constant depends on $K$.
\newline

Finally, we estimate $u_2$ on $B_x/16$: for any $z\in B_x/16$,
\begin{equation*}
|u_2(z)|
\siml K^{-\eta} \fint_{KB} |u_2|\dm
\siml K^{-\eta} \fint_{8KB} \NNt (u\UN_{KB})\dmu
\siml K^{-\eta} \MM_\mu \NNt (u\UN_{KB}) (\xi), 
\end{equation*}
where we used the Hölder regularity estimate \eqref{eq:boundary_holder} and Lemma \ref{lemma:uNu}, and the fact that $u_2\leq u$.
\newline

Now, we can piece together our three estimates, which yields
\begin{equation*}
\NNt^{(1,1/16)} u(\xi)
\leq  C_K \p{\AAt^{(2,1/2)}(\delta^2f)(\xi)
+ \MM_{\mu,q}\AAt(\delta^2f)(\xi) }
+ \frac{C}{K^\eta}\MM\NNt (u\UN_{KB}) (\xi) \for \xi\in\dOmega.
\end{equation*}
This implies, using the $L^p$-boundedness of $M_{\mu,q}$ when $p>q$ and the equivalence under change of parameters (Lemma \ref{lemma:eqvlce_NtAt}):
\begin{equation}\label{eq:DPD:Ntu}
\|\NNt u\|_p
\leq C'_K\|\AAt(\delta^2f)\|_p + \frac{C'}{K^\eta}\|\NNt(u\UN_{KB})\|_p.
\end{equation}
Thanks to the boundary Hölder estimate \ref{eq:boundary_holder}, we have that $u\UN_{KB}\incl L^\infty_c(\bOmega,m)$. This implies by a simple estimate that $u\in L^\infty_c(\dOmega,\mu)\incl L^p(\dOmega,\mu)$. Then, \eqref{eq:DPD:Ntu} give us that $\|\NNt u\|_p$ is finite and
\begin{equation*}
\|\NNt u\|_p
\leq C'_K\|\AAt(\delta^2f)\|_p + \frac{C'}{K^\eta}\|\NNt u\|_p.
\end{equation*}
Thus, we can choose $K$ big enough with respect to $C'$ so that $\|\NNt u\|_p \siml \|\AAt(\delta^2f)\|_p$, which is the desired estimate.
\newline

\textbf{Proof of $(i)\Rightarrow (iv)$.}
The proof for $-\div(wF)$ is strongly analogous to the one for $wf$; we will only highlight the small adjustments required. Keeping the same notations, we decompose $F:=F_0+F_1+F_2$ in exactly the same way, giving us a corresponding decomposition of the solution $u=u_0+u_1+u_2$. 
\newline 

For the estimate on $u_0$, we only need to notice that 
\begin{equation*}
\int_{B_x/8} |\grad u_0|^2\dm 
\siml \int_\Omega A\grad u_0\cdot\grad u_0\dm
=\int_{B_x/8} F\cdot\grad u_0\dm 
\end{equation*} 
implies
\begin{equation*}
\pfrac{\int_{B_x/8} |\grad u_0|^2\dm}
\siml \pfrac{\int_{B_x/8} |F|^2\dm }.
\end{equation*}
We can now use the boundary Poincaré inequality and conclude with the same calculation. For $u_1$, the only change is that we now use \eqref{eq:estimee_green_grad} instead of \eqref{eq:estimee_green}. Finally, no change is needed to estimate $u_2$, and we can conclude in the exact same manner to finish the proof.
\newline

\textbf{Proof of $(iii)\Rightarrow(i)$.}
To obtain the result, it is enough to prove the estimate \eqref{eq:estimee_green} for $\kappa_1$. Consider $B=B(\xi,r)$ centered on $\dOmega$ and $x\in B$ satisfying $\delta(x)>\kappa_1 r$. Using the duality estimate \eqref{eq:dual_Ntsup},
\begin{align*}
\|\NNt \left(\frac{\rho}{\delta}G(x,\cdot)\UN_{B\setminus B_x/2}\right)\|_\pp
&\siml \sup_{\|\AAt(\ft)\|_p\leq 1} \int_\Omega \frac{1}{\delta(y)^2}G(x,y)\UN_{B\setminus B_x/2}(y) \ft(y)\dm(y)\\
&\siml \sup_{\|\AAt(\delta^2f)\|_p\leq 1} \int_\Omega G(x,y)\UN_{B\setminus B_x/2}(y) f(y)\dm(y),
\end{align*}
We know thanks to Theorem \ref{th:densite_At} that $L^\infty_c(\Omega)$ is dense in $\At^p(\Omega)$, and $G\geq 0$, so the inequality is still true if the supremum on the right hand side only runs over functions $f\geq 0$ in $L^\infty_c(\Omega)$ such that $\|\AAt (\delta^2f)\|_p\leq 1$; take one of these functions $f$. Then

\begin{equation*}
\int_\Omega G(x,y)\UN_{B\setminus B_x/2}(y) f(y)\dm(y)
=u(x),
\end{equation*}
where $u$ is the solution in $\Omega$ of 
\begin{equation*}
\begin{cases}
Lu=wf\UN_{B\setminus B_x/2} &\texteq{in } \Omega\\
u=0 &\texteq{on } \dOmega.
\end{cases}
\end{equation*}
Since $Lu=0$ in $B_x/2$, we can use the Moser estimate \eqref{eq:int_moser} and obtain  
\begin{equation*}
u(x)
\siml\fint_{B_x/2} u\dm
\leq \pfrac{\fint_{B_x/2} u^2\dm}
\leq \NNt u(\zeta) \for \zeta\in 2B_x\cap\dOmega,
\end{equation*}
and since $(PD_p)$ is solvable when $F=0$, 
\begin{equation*}
u(x)
\siml \fint_{2B_x\cap\dOmega} \NNt u\dmu
\leq \frac{1}{\mu(2B_x)^{1/p}} \Np{\NNt u}
\siml \frac{1}{\mu(B)^{1/p}} \Np{\AAt (\delta^2f)}
\leq \frac{1}{\mu(B)^{1/p}},
\end{equation*}

which is the desired estimate.
\newline

\textbf{Proof of $(iv)\Rightarrow(i)$.}
The proof is the same as the proof of $(iii)\Rightarrow(i)$ except the fact that we use \eqref{eq:estimee_green_grad} instead of \eqref{eq:estimee_green} to obtain that $(D_p)$ is solvable.
\newline

\textbf{Proof of $(vi)\Rightarrow(iv)$.} This proof and the following ones are the same as in Theorem 1.22 in \cite{mourgoglou_solvability_2023}. Take $F\in L^\infty_c(\Omega,m)$, and consider $u\in W_0$ solution to $Lu=-\div (wF)$. using the duality estimate \eqref{eq:dual_Ntsup}, 
\begin{equation*}
\Np{\NNt u}
\siml \sup_{\|\AAt\tilde h\|_\pp\leq 1} \int_\Omega u\tilde h\frac{\dm}{\rho\delta} 
= \sup_{\Npp{\AAt(\rho\delta h)}\leq 1} \int_\Omega uh\dm. 
\end{equation*}
By density (Theorem \ref{th:densite_At}), it is enough to estimate the term in the supremum for any $h\in L^\infty_c(\Omega,m)$ such that $\Npp{\AAt(\rho\delta h)}\leq 1$. Take one such $h$, and consider $u^*\in W_0$ solution to $L^*u^*=wh$. Using the definition of weak solutions:
\begin{equation*}
\int_\Omega uh \dm
=\int_\Omega A^T\grad u^*\cdot\grad u \dm
=\int_\Omega A\grad u\cdot\grad u^*\dm
=\int_\Omega F\cdot\grad u^*\dm,
\end{equation*}
and we can estimate:
\begin{equation*}
\int_\Omega F\cdot\grad u^*\dm
\siml \Np{\AAt(\delta F)} \Npp{\NNt(\rho\grad u^*)}
\siml \Np{\AAt(\delta F)} \Npp{\AAt(\rho\delta h)}
\siml \Np{\AAt(\delta F)},
\end{equation*}
where we used the solvability of $(PD^*_\pp)$ for $wf$. This gives us that $\Np{\NNt u}\siml \Np{\AAt(\delta F)}$.
\newline

\textbf{Proof of $(iv)\Rightarrow(v)$.} If we take $f\in L^\infty_c(\Omega,m)$ and $u^*\in W_0$ solution to $L^*u^*= wf$, we can again use duality:

\begin{equation*}
\Npp{\NNt(\rho\grad u^*)}
\siml \sup_{\Np{\AAt(\delta H)}\leq 1} \int_\Omega \grad u^*\cdot H\dm.
\end{equation*}
If we take $H\in L^\infty_c(\Omega,m)$ such that $\Np{\AAt(\delta H)}\leq 1$ and consider $u\in W_0$ solution to $Lu=-\div(wH)$, we can estimate in the same way to obtain

\begin{equation*}
\int_\Omega \grad u^*\cdot H\dm
=\int_\Omega fu\dm
\siml \Npp{\AAt(\rho\delta f)}\Np{\NNt u}
\siml \Npp{\AAt(\rho\delta f)}\Np{\AAt(\delta H)}
\siml \Npp{\AAt(\rho\delta f)}.
\end{equation*}
From this we conclude that $(PR^*_\pp)$ is solvable with $F=0$.
\newline

Now, let us prove that $(PR^*_\pp)$ is solvable for $-\div(wF)$. Take $F\in L^\infty_c(\Omega,m)$, and consider $u^*\in W_0$ solution to $L^*u^*=-\div(wF)$. Again, by duality,
\begin{equation*}
\Npp{\NNt(\rho\grad u^*)}
\siml \sup_{\Np{\AAt(\delta H)}\leq 1} \int_\Omega \grad u^*\cdot H\dm.
\end{equation*}
Let us take $H\in L^\infty_c(\Omega,m)$ such that $\Np{\AAt(\delta H)}\leq 1$ and consider $u\in W_0$ solution to $Lu=-\div(wH)$. Then:
\begin{equation*}
\int_\Omega \grad u^*\cdot H\dm
=\int_\Omega F\cdot \grad u\dm
\siml \Npp{\AAt(\rho F)} \Np{\NNt(\delta\grad u)}.
\end{equation*}

Let us show that $\Np{\NNt(\delta\grad u)}\siml 1$; then we will have our result. Take $\xi\in\dOmega$ and $x\in\gamma(\xi)$. Using the Caccioppoli estimate \eqref{eq:int_caccio}, we can obtain
\begin{align*}
\pfrac{\fint_{B_x/16} |\delta\grad u|^2\dm}
&\siml \delta(x)\pfrac{\fint_{B_x/16} |\grad u|^2\dm}\\
&\siml \pfrac{\fint_{B_x/8} |u|^2\dm} + \delta(x)\pfrac{\fint_{B_x/8} |H|^2\dm}.
\end{align*}
We estimate the second term:
\begin{align*}
\delta(x)\pfrac{\fint_{B_x/8} |H|^2\dm}
\siml \pfrac{\fint_{B_x/8} |\delta H|^2\dm}
&\siml \int_{B_x/8} \pfrac{\fint_{B_x/8} |\delta H|^2\dm} \frac{\dm(y)}{m(B_y)}\\
&\siml \int_{B_x/8} \pfrac{\fint_{B_y/2} |\delta H|^2\dm} \frac{\dm(y)}{m(B_y)}\\
&\siml \AAt^{(2,1/2)}(\delta H)(\xi).
\end{align*}
This give us the estimate
\begin{equation*}
\NNt^{(1,1/16)} (\delta\grad u) (\xi)
\siml \NNt^{(1,1/8)} u (\xi) + \AAt^{(2,1/2)}(\delta H)(\xi) \for \xi\in\dOmega;
\end{equation*}
thus, since $(PD_p)$ is solvable with $f=0$,
\begin{equation*}
\Np{\NNt(\delta\grad u)}
\siml \Np{\NNt u} + \Np{\AAt(\delta H)}
\siml \Np{\AAt(\delta H)}
\leq 1,
\end{equation*}
which concludes the proof.

\subsection{Existence of solutions with general data}

We conclude this paper by proving the existence of solutions with general boundary data $g\in L^p(\dOmega,\dmu)$ and general interior data $f$ and $F$ belonging to suitable tent spaces.

\begin{corollary}\label{cor:DPD_general}
Assume that (H1)-(H6) are satisfied. Let $1<p<\infty$. If any of the equivalent conditions of Theorem \ref{th:eqvlce_DPD} holds, then there exists $C>0$ such that for any functions $g\in L^p(\dOmega,\mu)$, $f\in L^2\loc(\Omega,m)$ satisfying $\|\AAt(\delta^2 f)\|_p<\infty$ and $F\in L^2\loc(\Omega,m)^n$ satisfying $\|\AAt(\delta F)\|_p<\infty$, there exists a weak solution $u\in W\loc(\Omega)$ to
\begin{equation}\label{eq:Poisson_Dirichlet_géné}
\begin{cases}
Lu=wf-\div(wF) &\texteq{in }\Omega\\
u=g &\texteq{on }\dOmega
\end{cases}
\end{equation}
satisfying the estimate
\begin{equation}\label{eq:pd_estimate}
\Lpd{\NNt u} \leq C\cro{\Lpd{g} + \Lpd{\AAt(\delta^2f)}+ \Lpd{\AAt(\delta F)}}.
\end{equation}
\end{corollary}

\begin{proof}
The proof follows from Theorem \ref{th:eqvlce_DPD} and a density argument. By linearity it suffices to prove the result in the cases $f=F=0$ and $g=0$. Let us begin with the first case and consider $g\in L^p(\dOmega,\mu)$. By density of $\CC_c(\dOmega)$ in $L^p(\Omega,\mu)$, there is a sequence $(g_k)$ in $\CC_c(\dOmega)$ such that $g_k\to g$ in $L^p(\dOmega)$. For each $k$, let $u_k$ be the corresponding solution given by Theorem \ref{th:elliptic_measure^}; we then have
\begin{equation*}
u_k(x)-u_\ell(x)=\int_\dOmega (g_k-g_\ell)d\omega^x \for x\in\Omega \texteq{ and }k,l\in\N.
\end{equation*}
Any of the equivalent conditions of Theorem \ref{th:eqvlce_DPD} holds, so $(D_p)$ is solvable; hence 
\begin{equation} \label{ukCauchy}
\|\NNt(u_k-u_\ell)\|_p \siml \|g_k-g_\ell\|_p,
\end{equation}
which implies that $(u_k)$ is a Cauchy sequence in $\Nt^p(\Omega)$, which is complete by Lemma \ref{lemma:tent_space_complete}, so it has a limit $u$ in $\Nt^p(\Omega)$. Since $\Nt^p(\Omega) \subset L^2\loc(\Omega,m)$ is a continuous injection, $(u_k)$ is a Cauchy sequence in $L^2\loc(\Omega,m)$ converging to $u$ in $L^2\loc(\Omega,m)$, which implies thanks to the Caccioppoli inequality \eqref{eq:int_caccio} that $(\grad u_k)$ is a Cauchy sequence in $L^2\loc(\Omega,m)$. Using Proposition \ref{prop:W_complete}, we obtain that $u\in W\loc(\Omega)$ with $u_k\to u$ in $L^1\loc(\Omega,m)$ and $\grad u_k\to u\in L^2\loc(\Omega,m)$; it is then easy to check that $Lu=0$ weakly. Finally, the fact that $u_k\to u$ in $\N^p(\Omega)$ and that $(g_k\to g$ in $L^p(\dOmega,\mu)$ implies that, up to a subsequence and $\mu$-a.e.,
\begin{equation}\label{eq:DPD_general:cv_pp}
\sup_{y\in\gamma_2(\xi)}\pfrac{\fint_{B_y/2}|u_k-u|^2\dm}\longrightarrow 0 
\texteq{ and  } g_k\longrightarrow g
\formuae \xi\in\dOmega,
\end{equation}
and for any $\xi\in\dOmega$ and $x\in \gamma(\xi)$, 
\begin{align*}
\fint_{B_x/2} |u-g(\xi)|\dm
\leq& \sup_{y\in\gamma_2(\xi)}\pfrac{\fint_{B_y/2}|u_k-u|^2\dm} + \fint_{B_x/2}|u_k-g_k(\xi)|\dm\\
&+ |g_k-g|(\xi).
\end{align*}
using \eqref{eq:DPD_general:cv_pp} and the fact that $u_k(y)\to g_k(\xi)$ when $y\to\xi$ thanks to Theorem \ref{th:elliptic_measure^}, we deduce that
\begin{equation*}
\fint_{B_x/2} |u-g(\xi)|\dm \cvstack{x\to\xi}{x\in\gamma(\xi)} 0 \formuae \xi\in\dOmega.
\end{equation*}
This give us the result for $f=F=0$.
\newline

To prove the result for $g=0$, consider $f\in L^2\loc(\Omega,m)$ such that $\delta^2f\in\At^p(\Omega)$ and $F\in L^2\loc(\Omega,m)^n$ such that $\delta F\in\At^p(\Omega)^n$. By Theorem \ref{th:densite_At}, we can take sequences $(f_k)$ in $L^\infty_c(\Omega,m)$ and $(F_k)$ in $L^\infty_c(\Omega)^n$ such that $\delta^2 f_k\to\delta^2 f$ in $\At^p(\Omega)$ and $\delta F\to\delta F$ in $\At^p(\Omega)^n$, which implies in particular that $f_k\to f$ in $L^2\loc(\Omega,m)$ and $F_k\to F$ in $L^2\loc(\Omega,m)^n$. For each $k$, let $u_k$ be the corresponding solution in $W$ of \ref{eq:Poisson_Dirichlet_géné} given by Theorem \ref{th:solution_W}. We have
\begin{equation*}
	\begin{cases}
		L(u_k-u_\ell)=w(f_k-f_\ell) -\div\p{w(F_k-F_\ell)}& \texteq{in } \Omega\\
		u_k-u_\ell=0 & \texteq{on } \dOmega
	\end{cases}
	\for k,l\in\N.
\end{equation*}
Any of the equivalent conditions of Theorem \ref{th:eqvlce_DPD} holds, so $(PD_p)$ is solvable; hence 
\begin{equation}
`	\|\NNt(u_k-u_\ell)\|_p \siml \|\AAt(\delta^2f_k-\delta^2f_\ell)\|_p + \|\AAt(\delta F_k-\delta F_\ell)\|_p \for k,\ell\in\N,
\end{equation}
implying that $(u_k)$ is a Cauchy sequence in $\Nt^p(\Omega)$. Using the same reasoning as for the previous case, the fact that $f_k\to f$ in $L^2\loc(\Omega,m)$ and $F_k\to F$ in $L^2\loc(\Omega,m)^n$ and the Caccioppoli inequality with source terms \eqref{eq:int_caccio}, we get that $(u_k)$ has a limit $u\in W\loc(\Omega)$ that is a weak solution of $Lu=f-\div(wF)$. Finally, for any $\xi\in\dOmega$ and $x\in\gamma(\xi)$,
\begin{equation*}
	\fint_{B_x/2}|u|\dm
	\leq \sup_{y\in\gamma_2(\xi)}\pfrac{\fint_{B_y/2}|u_k-u|^2\dm} + \fint_{B_x/2}|u_k|\dm,
\end{equation*}
$u_k\in W_0(\Omega)$ and up to extraction
\begin{equation*}
	\sup_{y\in\gamma_2(\xi)}\pfrac{\fint_{B_y/2}|u_k-u|^2\dm}\longrightarrow 0 \formuae \xi\in\dOmega,
\end{equation*}
so
\begin{equation*}
\fint_{B_x/2} |u|\dm \cvstack{x\to\xi}{x\in\gamma(\xi)} 0 \formuae \xi\in\dOmega.
\end{equation*}
\end{proof}

Of course, we have the same result for the Poisson-regularity problem.

\begin{corollary}
Assume that (H1)-(H6) are satisfied. Let $1<p<\infty$. If any of the equivalent conditions of Theorem \ref{th:eqvlce_DPD} holds, then there exists $C>0$ such that for any functions $f\in L^2\loc(\Omega,m)$ satisfying $\|\AAt(\rho\delta f)\|_\pp<\infty$ and $F\in L^2\loc(\Omega,m)^n$ satisfying $\|\AAt(\rho F)\|_\pp<\infty$, there exists a weak solution $u\in W\loc(\Omega)$ to
\begin{equation}
\begin{cases}
Lu=wf-\div(wF) &\texteq{in }\Omega\\
u=0 &\texteq{on }\dOmega
\end{cases}
\end{equation}
satisfying the estimate
\begin{equation}
\Lpd[\pp]{\NNt(\rho\grad u)} \leq C\cro{\Lpd[\pp]{\AAt(\rho\delta f)}+ \Lpd[\pp]{\AAt(\rho F)}}.
\end{equation}
\end{corollary}
\begin{proof}
The proof is completely analogous to the proof of Corollary \ref{cor:DPD_general}. 
\end{proof}

%%%%%%%%%%%%%%%%%%%%%%%%%%%%%%%%%%%%%%%%%%%%%%%%%%%%%%%%%%%%%%%%%%%%%
%Appendix
%%%%%%%%%%%%%%%%%%%%%%%%%%%%%%%%%%%%%%%%%%%%%%%%%%%%%%%%%%%%%%%%%%%%%

\appendix
\renewcommand{\thetout}{\Alph{section}.\arabic{tout}}
\renewcommand{\theequation}{\Alph{section}.\arabic{equation}}

\section{Appendix: results on tent spaces}\label{s:appendix}

We detail here some fundamental results on tent spaces. All of the proofs below are an adaptation to our needs of arguments found in \cite{coifman_new_1985}. We feel the need to rewrite the arguments due our setting is getting far from the case $\Omega=\R^n \times (0,\infty)$ that was considered in \cite{coifman_new_1985}. 
\newline

In this appendix, we only need to assume that {\bf $m_{|\Omega}$ and $\mu$ are doubling} (with constants $C_m'$ and $C_\mu$); in particular, we do not assume the corkscrew point condition in $\Omega$.

\begin{definition}
Take $\alpha>0$ and $0<\lambda\leq 1/2$. Let us define:
\begin{itemize}
\item for any measurable function $f:\Omega\to\R$,
\begin{equation*}
\NN^{(\alpha)} f (\xi):= \sup_{\gamma_\alpha(\xi)} |f|, \for \xi\in\dOmega.
\end{equation*}

\item for any function $f\in L^2\loc(\Omega,m)$,
\begin{equation*}
\NNt^{(\alpha,\lambda)} f (\xi):= \sup_{\gamma_\alpha(\xi)} \pfrac{\fint_{\lambda B_x} |f|^2\dm}, \for \xi\in\dOmega.
\end{equation*}

\item for any measurable function $g:\Omega\to\R$,
\begin{equation*}
\AA^{(\alpha)} g(\xi):= \int_{\gamma_\alpha(\xi)} |g(x)| \frac{\dm(x)}{m(B_x)}, \for \xi\in\dOmega.
\end{equation*}

\item for any function $g\in L^2\loc(\Omega,m)$,
\begin{equation*}
\AAt^{(\alpha,\lambda)} g(\xi):= \int_{\gamma_\alpha(\xi)} \pfrac{\fint_{\lambda B_x} |f|^2\dm} \frac{\dm(x)}{m(B_x)}, \for \xi\in\dOmega.
\end{equation*}
\end{itemize}
\end{definition}

For any $\phi\in L^1\loc(\dOmega,\mu)$, recall that we write
\begin{equation*}
\MM_\mu f(\xi) := \sup_{r>0} \fint_{B(x,r)} f\dmu \for \xi\in\dOmega,
\end{equation*}
the centered Hardy-Littlewood maximal function.

\begin{lemma}\label{lemmaA:eqvlce_N}
Let $1\leq p\leq \infty$. The $L^p$ norms of $\NN^{(\alpha)}$ and are equivalent under change of $\alpha$, i.e. for any $\alpha,\alpha'> 0$,
\begin{equation*}
\|\NN^{(\alpha)}f\|_p \leq C_{\alpha,\alpha'} \|\NN^{(\alpha')}f\|_p,
\end{equation*}
for any $f$ measurable, where $C_{\alpha,\alpha'}>0$ depends only on $C_\mu$, $\alpha$ and $\alpha'$; furthermore, if $1\leq \alpha'<\alpha$, $C$ depends only on $C_\mu$ and $\alpha/\alpha'$. Likewise, the $L^p$ norms of $\NNt^{(\alpha,\lambda)}$ are equivalent under change of $\alpha$ and $\lambda$, i.e. for any $\alpha,\alpha'>0$ and any $0<\lambda,\lambda'\leq 1/2$,
\begin{equation*}
\|\NNt^{(\alpha,\lambda)} f\|_p \leq C_{\alpha,\alpha',\lambda/\lambda'} \|\NNt^{(\alpha',\lambda')} f\|_p
\end{equation*}
for any $f\in L^2\loc(\Omega,m)$, where $C_{\alpha,\alpha',\lambda/\lambda'}>0$ depends only on $C_\mu, C_m',\alpha,\alpha'$ and $\lambda/\lambda'$.
\end{lemma}

The same result holds for $\AA$ and $\AAt$, but it will be easier to prove it later from this lemma and duality  (Lemma \ref{lemmaA:eqvlce_A}).

\begin{proof}
Let us prove the first estimate. Take $f\geq 0$ measurable, and $t>0$; we write $E^\alpha_t:=\{\xi\in\dOmega\tq \NN^{(\alpha)} f(\xi)>t \}$. to prove the lemma, it is enough to show that $\mu(E^\alpha_t)\siml \mu(E^{\alpha'}_t)$, with a constant independent from $f$ and $t$.
\newline

Take $\xi\in E^\alpha_t$: $\sup_{\gamma_\alpha(\xi)} f > t$. There exists $x\in\gamma_\alpha(\xi)$ such that $f(x)>t$, so if $\zeta\in (1+\alpha')B_x \cap \partial \Omega$, we have $x\in \gamma_{\alpha'}(\zeta)$, or 
\[\NN^{(\alpha')}f(\zeta) > t \qquad \text{ for } \zeta\in (1+\alpha')B_x \cap \partial \Omega.\]
Take $B=B_{\xi,t}:= B(\xi,(2+\alpha + \alpha')\delta(x))$ so that  $B\incl (1+\alpha')B_x$, and we get
\begin{equation*}
\fint_{B\cap \partial \Omega} \UN_{E_t^{\alpha'}}\dmu
=\frac{\mu(E_t^{\alpha'}\cap B)}{\mu(B\cap \partial \Omega)}
\geq \frac{\mu((1+\alpha')B_x)}{\mu(B \cap \partial \Omega)}
> \epsilon,
\end{equation*}
with a constant $\epsilon>0$ depending only $\alpha$, $\alpha'$ ($\alpha/\alpha'$ if we further assume $1\leq\alpha'<\alpha$) and $C_\mu$. This implies that $\MM_\mu\UN_{E_t^{\alpha'}} > \epsilon$, and finally by the Hardy–Littlewood maximal inequality
\begin{equation*}
\mu(E_t^\alpha)
\leq \mu\p{\{\xi\in\dOmega\tq \MM_\mu\UN_{E_t^{\alpha'}}(\xi) > \epsilon\}}
\siml \frac{1}{\epsilon}\mu(E_t^{\alpha'}) \for t>0,
\end{equation*}
which is enough to obtain the first estimate.
\newline

Now, Let us prove the equivalence for $\NNt$. By the first estimate, the equivalence holds under change of $\alpha$ when $\lambda$ is fixed,  hence it is enough to prove the equivalence for some fixed $\alpha,\alpha'$; we will take $\alpha=1$. Take $0<\lambda,\lambda'\leq 1/2$, $\xi\in\dOmega$ and $x\in\gamma(\xi)$. There exists an integer $N\geq 1$ independent from $x$ (but dependent on $\lambda/\lambda'$) and $y_1,\ldots,y_N$ in $\lambda B_x$ such that $\lambda B_x\incl \bigcup_{i=1}^N \lambda'B(y_i)$. Then for any $f\in L^2\loc(\Omega)$,

\begin{equation*}
\pfrac{\fint_{\lambda B_x} f^2\dm}
\siml \sum_{i=1}^N \pfrac{\fint_{\lambda'B(y_i)} f^2\dm}
\siml \sup_{y\in\gamma_4(\xi)} \pfrac{\fint_{\lambda'B_y}f^2\dm},
\end{equation*}
hence $\NNt^{(1,\lambda)} f\siml\NNt^{(4,\lambda')} f$, which is enough to conclude.
\end{proof}

We will note $\NN:=\NN^{(1)}$, $\NNt=\NNt^{(1,1/2)}$, $\AA:=\AA^{(1)}$ and $\AAt=\AAt^{(1,1/2)}$. For $1<p<\infty$, define $\N^p(\Omega)$ (resp. $\A^p(\Omega)$) as the vector space of measurable function $f:\Omega\to\R$ (resp. $g:\Omega\to\R$) such that $\|\NN f\|_p$ is finite (resp. $\|\AA g\|_p$ is finite), equipped with the obvious norms. Likewise, define $\Nt^p(\Omega)$ (resp. $\At^p(\Omega)$) as the vector space of function $f\in L^2\loc(\Omega,m)$ (resp. $g\in L^2\loc(\Omega,m)$) such that $\|\NNt f\|_p$ is finite (resp. $\|\AAt g\|_p$ is finite).
\newline

First, let us observe that these norms satisfy a version of Fatou's lemma:
\begin{lemma}[Fatou's lemma for tent spaces]
Take $1\leq p\leq\infty$ and a sequence $(f_k)$ of measurable non-negative functions. then
\begin{equation*}
\|T(\liminf f_k)\|_p\leq \liminf \|Tf_k\|_p,
\end{equation*}
whenever $T=\NN,\NNt,\AA$ or $\AAt$.
\end{lemma}
\begin{proof}
This is a direct consequence of Fatou's lemma and of the fact that
\begin{equation}\label{eqA:fatou}
\sup_{A}\liminf f_k \leq \liminf \sup_A f_k
\end{equation}
 for any set $A\incl\R^n$.
\end{proof}

\begin{theorem}\label{thA:Banach}
Take $1<p<\infty$. All tent spaces $\N^p(\Omega)$, $\Nt^p(\Omega)$, $\A^p(\Omega)$ and $\At^p(\Omega)$ are Banach spaces.
\end{theorem}

\begin{proof}
We will only prove the theorem for $\Nt^p(\Omega)$ since it is the result that we use in the main article, but the proofs for the other spaces use the same ideas. Take a Cauchy sequence $(f_k)$ in $\Nt^p(\Omega)$. First, for any $x\in\Omega$, $\xi\in 2B_x\cap\dOmega$ and $k,\ell\in\N$:
\begin{equation*}
\pfrac{\fint_{B_x/2}|f_k-f_\ell|^2\dm}
\leq \NNt(f_k-f_\ell)(\xi);
\end{equation*}
averaging over $\xi\in 2B_x\cap\dOmega$, we obtain
\begin{equation*}
\pfrac{\fint_{B_x/2}|f_k-f_\ell|^2\dm}
\leq \pfrac[p]{\fint_{2B_x}\NNt(f_k-f_\ell)^p\dmu}.
\end{equation*}
Thus, $(f_k)$ is a Cauchy sequence in $L^2\loc(\Omega,m)$: there exists $f\in L^2\loc(\Omega,m)$ such that $f_k\to f$ in $L^2\loc(\Omega,m)$, and we can extract a subsequence $(\ft_k)$ such that $\ft_k\to f$ a.e. on $\Omega$. Then, for any $k\in\N$ Fatou's lemma \eqref{eqA:fatou} applied to $|f_k-\ft_\ell|$ give us that
\begin{equation*}
\|\NNt(f_k-f)\|_p
\leq \liminf_\ell\|\NNt(f_k-\ft_\ell)\|_p,
\end{equation*}
which implies that $f\in\Nt^p(\Omega)$ and, passing to the limit as $k$ goes to infinity, that $f_k\to f$ in $\Nt^p(\Omega)$. 
\end{proof}

\begin{theorem}\label{thA:dual_intNA}
Let $\alpha>0$ and $0<\lambda\leq 1/2$. For any $f,g$ measurable,
\begin{equation}\label{eqA:dual_intNA}
\int_\Omega fg \frac{\dm}{\rho\delta} 
\leq C_\alpha \int_\dOmega (\NN^{(\alpha)}f) (\AA^{(\alpha)}g) \dmu,
\end{equation}
where $C_\alpha>0$ depends only on $C_m', C_\mu$ and $\alpha$. Likewise, for any $f,g\in L^2\loc(\Omega,m)$,
\begin{equation}\label{eqA:dual_int_NtAt}
\int_\Omega fg \frac{\dm}{\rho\delta} 
\leq C_{\alpha,\lambda} \int_\dOmega (\NNt^{(\alpha,\lambda)}f) (\AAt^{(\alpha,\lambda)}g) \dmu,
\end{equation}
where $C_{\alpha,\lambda}>0$ depends only on $C_m', C_\mu$, $\alpha$ and $\lambda$.
\end{theorem}

\begin{proof}
To prove the first estimate, notice that for any $h\geq 0$ measurable,
\begin{equation}\label{eqA:dual_intA}
\int_{\Omega} h \frac{\dm}{\rho\delta}
\simeq \int_{\dOmega} \AA^{(\alpha)}h \dmu; 
\end{equation}
indeed,
\begin{align*}
\int_\dOmega \AA^{(\alpha)} h \dmu
&=\int_\dOmega \int_{\gamma_\alpha(\xi)} h(x) \frac{\dm(x)}{m(B_x)} \dmu(\xi)\\
&\simeq\int_\Omega \int_{(1+\alpha)B_x\cap\dOmega} \dmu(\xi) \unsur{\mu(2B_x)}\unsur{\rho(x)\delta(x)} h(x)\dm(x)\\
&\simeq \int_\Omega h \frac{\dm}{\rho\delta},
\end{align*}
where we used Fubini's theorem and the doubling property of $m$ in the second line, and the doubling property of $\mu$ in the third one. Then, \eqref{eqA:dual_intA} applied to  $|fg|$ easily give us the result.
\newline

Next, to prove the second estimate, notice that for any $h\geq 0$ in $L^1\loc(\Omega,m)$,
\begin{equation*}
\int_\Omega \fint_{\lambda B_x} h(y)\dm(y) \frac{\dm(x)}{\rho(x)\delta(x)}
\geq \int_\Omega h(y) \int_{\frac{\lambda}{1+\lambda}B_y}\frac{1}{m(\lambda B_x)\rho(x)\delta(x)}\dm(x)\dm(y)
\simeq \int_\Omega h(y)\frac{\dm(y)}{\rho(y)\delta(y)},
\end{equation*}
where we used Fubini's theorem and the fact that $x\in\frac{\lambda}{1+\lambda}B_y$ implies $y\in\lambda B_x$ for the first estimate, and the doubling property of $m$ and $\rho$ for the second one.
Then, for any $f,g\in L^2\loc(\Omega,m)$,
\begin{align*}
\int_\Omega |fg|\frac{\dm}{\rho\delta}
\siml \int_\Omega \fint_{\lambda B_x} |fg| \dm \frac{\dm(x)}{\rho(x)\delta(x)}
&\leq \int_\Omega \pfrac{\fint_{\lambda B_x} f^2\dm} \pfrac{\fint_{\lambda B_x} g^2\dm} \frac{\dm(x)}{\rho(x)\delta(x)}\\
&\siml \int_\dOmega (\NNt^{(\alpha,\lambda)} f)(\AAt ^{(\alpha,\lambda)} g) \dmu,
\end{align*}
using \eqref{eqA:dual_intNA}. This give us the second estimate.
\end{proof}

Thanks to the H\"older inequality, we obtain the following corollary.
\begin{corollary}
Let $1< p<\infty$. We have
\begin{equation}\label{eqA:dual_prodNA}
\int_\Omega fg \frac{\dm}{\rho\delta} \leq C\|\NN f\|_p\|\AA g\|_{p'} \for f\in \N^p(\Omega), g\in \A^\pp(\Omega)
\end{equation}
and
\begin{equation}\label{eqA:dual_prodNtAt}
\int_\Omega fg \frac{\dm}{\rho\delta} \leq C\|\NNt f\|_p\|\AAt g\|_{p'} \for f\in \Nt^p(\Omega), g\in \At^\pp(\Omega),
\end{equation}
where $C>0$ depends only on $C_m'$ and $C_\mu$.
\end{corollary}

\begin{theorem}\label{thA:dual_A'N}
Let $1<p<\infty$. The dual space of $\A^p(\Omega)$ is homeomorphic to $\N^{p'}(\Omega)$ \textit{via} $f\mapsto\p{g\mapsto\int_\Omega fg\frac{\dm}{\rho\delta}}$. Likewise, the dual space of $\At^p(\Omega)$ is homeomorphic to $\Nt^{p'}(\Omega)$ \textit{via} the same homeomorphism. In particular,
\begin{equation}\label{eqA:dual_Nsup}
\|\NN f\|_{p'}
\leq C \sup_{\|\AA g\|_p\leq 1} \int_\Omega fg\frac{\dm}{\rho\delta}  \for f\in \N^{p'}(\Omega)
\end{equation}
and
\begin{equation}\label{eqA:dual_Ntsup}
\|\NNt f\|_{p'} \leq C \sup_{\|\AAt g\|_p\leq 1} \int_\Omega fg\frac{\dm}{\rho\delta} \for f\in \N^{p'}(\Omega),
\end{equation}
where $C>0$ depends only on $p,C_m'$ and $C_\mu$.

\end{theorem}

\begin{proof}
We begin by proving \eqref{eqA:dual_Nsup}. Note that it is enough to prove it only for $f\geq 0$, and take $f\geq 0$ in $\N^\pp(\Omega)$. By duality,
\begin{equation}
\|\NN f\|_{p'}
=\sup_{\|\phi\|_p\leq 1} \int_\dOmega (\NN f)\phi\dmu.
\end{equation}
We thus take $\phi\in L^p(\dOmega,\mu)$ such that $\|\phi\|_p\leq 1$; since $\NN f\geq 0$, we can also assume that $\phi\geq 0$. Furthermore, also by duality:
\begin{equation}
\NN f(\xi)
=\sup_{\gamma(\xi)} f
=\sup_{\|\KK(\xi,\cdot)\|_{L^1(\gamma(\xi))}\leq 1} \int_\Omega f(x)\KK(\xi,x) \dm(x)
\for \xi\in\dOmega,
\end{equation}
so there exists $\KK:\dOmega\times\Omega\to\R$ such that $\KK\geq 0$, $\KK(\xi,\cdot)=0$ outside of $\gamma(\xi)$, $\|\KK(\xi,\cdot)\|_1\leq 1$, and 
\begin{equation*}
\frac{1}{2} \NN f(\xi)
\leq \int_\Omega f(x)\KK(\xi,x)\dm(x)
\leq \NN f(\xi),
\for\xi\in\dOmega.
\end{equation*}

Then, using Fubini's theorem,
\begin{align*}
\int_\dOmega (\NN f)\phi \dmu
&\siml \int_\dOmega \phi(\xi)\int_{\gamma(\xi)} f(x)\KK(\xi,x)\dm(x) \dmu(\xi)\\
&\siml \int_\Omega f(x)\int_{2B_x\cap\dOmega}\phi(\xi)\KK(\xi,x)\dmu(\xi)\dm(x)\\
&\siml \int_\Omega fg\frac{\dm}{\rho\delta},
\end{align*}
where $g(x)=\rho(x)\delta(x)\int_{2B_x\cap\dOmega} \phi(\xi)\KK(\xi,x)\dmu(\xi)$, for $x\in\Omega$. Next, we prove that $\|\AA g\|_p\siml 1$. For any $k\in\Z$, note $\Omega_k=\{x\in\Omega\tq 2^{k-1}\leq \delta(x)<2^k\}$. Then, for any $\zeta\in\dOmega$,
\begin{align*}
\AA g(\zeta)
&= \int_{\gamma(\zeta)} \rho(x)\delta(x) \int_{2B_x\cap\dOmega} \phi(\xi)\KK(\xi,x)\dmu(\xi)\frac{\dm(x)}{m(B_x)}\\
&\siml \sum_{k\in\Z} \int_{\Omega_k\cap\gamma(\zeta)} \frac{1}{\mu(2B_x)} \int_{2B_x\cap\dOmega} \phi(\xi)\KK(\xi,x)\dmu(\xi)\dm(x)\\
&\siml \sum_{k\in\Z} \frac{1}{\mu(B(\zeta,2^k))} \int_{\Omega_k\cap\gamma(\zeta)} \int_{2B_x\cap\dOmega} \phi(\xi)\KK(\xi,x)\dmu(\xi)\dm(x)\\
&\siml \sum_{k\in\Z} \frac{1}{\mu(B(\zeta,2^k))} \int_{B(\zeta,2^{k+2})\cap\dOmega} \phi(\xi)\kappa_k(\xi)\dmu(\xi),
\end{align*}
with $\kappa_k(\xi)=\int_{\Omega_k} \KK(\xi,\cdot)\dm$ for $\xi\in\dOmega$, where we used the doubling properties of $\mu$ for the second estimate and Fubini for the third one. Note that $\sum_{k\in\Z} \kappa_k =1$. In order to estimate $\|\AA g\|_p$, we use duality once again: for any $\psi\geq 0$ in $L^\pp(\dOmega,\mu)$ such that $\|\psi\|_\pp\leq 1$, 

\begin{align*}
\int_\dOmega (\AA g)\psi\dmu
&\siml \int_\dOmega \psi(\zeta) \sum_{k\in\Z} \frac{1}{\mu(B(\zeta,2^k))} \int_{B(\zeta,2^{k+2})\cap\dOmega} \phi(\xi)\kappa_k(\xi)\dmu(\xi) \dmu(\zeta)\\
&\siml \int_\dOmega \phi(\xi) \sum_{k\in\Z} \kappa_k(\xi) \frac{1}{\mu(B(\xi,2^k))} \int_{B(\xi,2^{k+2})\cap\dOmega} \psi(\zeta)\dmu(\zeta) \dmu(\xi)\\
&\siml \int_\dOmega \phi(\xi)\sum_{k\in\Z} \kappa_k(\xi) \MM_\mu\psi(\xi) \dmu(\xi)\\
&\siml \int_\dOmega\phi\MM_\mu\psi\dmu
\siml \|\phi\|_p\|\MM_\mu\psi\|_\pp
\siml \|\phi\|_p\|\psi\|_\pp
=1,
\end{align*}
where we used Fubini and the doubling properties of $\mu$ for the second estimate, and the $L^p$-boundedness of $\MM_\mu$ for the last one. This implies that $\|\AA g\|_p\siml 1$, as we desired, and proves \eqref{eqA:dual_Nsup}.
\newline

Now, let us prove \eqref{eqA:dual_Ntsup}. It is enough to prove it for $f\geq 0$ in $\Nt^\pp(\Omega)$. By \eqref{eqA:dual_Nsup}, 

\begin{equation}
\|\NNt f\|_\pp
\siml \sup_{\|\AA g\|_p\leq 1} \int_\Omega \pfrac{\fint_{B_x/2} f^2\dm} g(x) \frac{\dm(x)}{\rho(x)\delta(x)}.
\end{equation}
Take $g\geq 0$ in $A^p(\Omega)$ such that $\|\AA g\|_p\leq 1$. Furthermore, also by duality,

\begin{equation}
\pfrac{\int_{B_x/2} f^2\dm}
= \sup_{\|H(x,\cdot)\|_{L^2(B_x/8)}\leq 1} \int_{B_x/2} f(y)H(x,y)\dm(y) \for x\in\Omega,
\end{equation}
so there exists $H:\Omega\times\Omega\to\R$ such that $H\geq 0$, $H(x,\cdot)=0$ outside of $B_x/2$, $\|H(x,\cdot)\|_2\leq 1$ and

\begin{equation}
\frac{1}{2}\pfrac{\int f^2\dm} 
\leq \int_{B_x/2} f(y)H(x,y)\dm(y) 
\leq \pfrac{\int f^2\dm} \for x\in\Omega.
\end{equation}
Then,
\begin{align*}
\int_\Omega \pfrac{\fint_{B_x/2} f^2\dm} g(x)\frac{dm(x)}{\rho(x)\delta(x)}
&\siml \int_\Omega g(x)\frac{1}{m(B_x)^{1/2}}\int_{B_x/2} f(y)H(x,y)\dm(y)\frac{dm(x)}{\rho(x)\delta(x)}\\
&\siml \int_\Omega f(y)\int_{\Omega} g(x)H(x,y)\frac{\dm(x)}{m(B_x)^{1/2}} \frac{\dm(y)}{\rho(y)\delta(y)}
= \int_\Omega f\gt\frac{\dm}{\rho\delta},
\end{align*}
with $\gt(y):=\int_{\Omega} g(x)H(x,y)\frac{\dm(x)}{m(B_x)^{1/2}}$ for $y\in\Omega$, where we used Fubini's theorem and the doubling properties of $\rho$ for the second estimate. Now, let us show that $\|\AAt\gt\|_p\siml 1$. Just as above, there exists $K:\Omega\times\Omega\to\R$ such that $K\geq 0$, $K(z,\cdot)=0$ outside of $B_z/2$, $\|K(z,\cdot)\|_2\leq 1$ and 

\begin{equation*}
\pfrac{\int_{B_z/2} \p{\int_{\Omega} g(x)H(x,y)\frac{\dm(x)}{m(B_x)^2}}^2\dm(y)}
\simeq \int_{B_z/2} K(z,y)\int_{\Omega} g(x)H(x,y)\frac{\dm(x)}{m(B_x)^2} \dm(y)
\end{equation*}
for $z\in\Omega$. Then, for $\xi\in\dOmega$,

\begin{align*}
\AAt\gt(\xi)
&=\int_{\gamma(\xi)} \pfrac{\fint_{B_z/2} \p{\int_{\Omega} g(x)H(x,y)\frac{\dm(x)}{m(B_x)^{1/2}}}^2\dm(y)} \frac{\dm(z)}{m(B_z)}\\
&\siml \int_{\gamma(\xi)} \pfrac{\int_{B_z/2} \p{\int_{\Omega} g(x)H(x,y)\frac{\dm(x)}{m(B_x)^2}}^2\dm(y)}\dm(z)\\
&\siml \int_{z\in\gamma(\xi)} \int_{y\in B_z/2} K(z,y)\int_{x\in \Omega} g(x)H(x,y)\frac{\dm(x)}{m(B_x)^2} \dm(y) \dm(z)\\
&\siml \int_{x\in\gamma_{16}(\xi)} g(x) \int _{z\in 4B_x\cap\Omega} \int_{y\in \Omega} K(z,y)H(x,y)\dm(y)\dm(z)\frac{\dm(x)}{m(B_x)^2}\\
&\siml \int_{\gamma_{16}(\xi)} g(x) \int _{4B_x\cap\Omega} \|K(z,\cdot)\|_2\|H(x,\cdot)\|_2\dm(z)\frac{\dm(x)}{m(B_x)^2}\\
&\siml \int_{\gamma_{16}(\xi)} g(x) \frac{\dm(x)}{m(B_x)}
= \AA^{(10)} g(\xi);
\end{align*}
Where we used the fact that $H(x,y)=0$ if $|x-y|>\delta(x)$ and the doubling properties of $m$ for the first estimate, the same property of $H$ and Fubini for the third one, and H\"older's inequality for the fourth one. This implies that $\|\AAt \gt\|_p\siml \|\AA g\|_p \leq 1$, and proves \eqref{eqA:dual_Ntsup}.
\newline 

Finally, let us prove that $\A^p(\Omega)'$ is homeomorphic to $\N^\pp(\Omega)$; the characterization of $\At^p(\Omega)'$ can be proven using the exact same arguments. First, remark that \eqref{eqA:dual_prodNA} gives us that 
\begin{align*}
\N^\pp(\Omega) &\longrightarrow \A^p(\Omega)'\\
f&\longmapsto \p{g\mapsto\int fg \frac{\dm}{\rho\delta}}
\end{align*}
Is well-defined. Next, take $\ell\in \A^p(\Omega)'$; for any compact set $K$ of $\Omega$, there exists $C_K>0$ such that

\begin{equation*}
\|\AA g\|_p \leq C_K \|g\|_{L^1(K)}, \for g\in L^1(K,m).
\end{equation*}
By Riesz's theorem, there exists a unique $F_K\in L^p(K,m)$ such that $\ell(g)=\int_K F_K g \dm$ for any $g\in L^1(K,m)$. By taking $f=\rho\delta F_K$ on any compact $K$ of $\Omega$, we then obtain a well-defined and unique function $f\in L^{\infty}\loc(\Omega,m)$ such that 
\begin{equation}\label{eqA:dual_identification}
\ell(g)=\int_{\Omega} fg\frac{\dm}{\rho\delta},\for g\in L^1_c(\Omega,m).
\end{equation}
Let us prove that $\|\NN f\|_\pp\siml\|\ell\|$. Take a compact $K$ in $\Omega$, and write $f_K:=f\UN_K$. Then $f_K\in L^\infty_c(\Omega,m)\incl N^\pp(\Omega)$, and
\begin{equation}
\|\NN f_K\|_\pp
\siml \sup_{\|\AA g\|_p\leq 1} \int_\Omega f_Kg\frac{\dm}{\rho\delta}
= \sup_{\|\AA g\|_p\leq 1} \ell(g\UN_K)
\leq \|\ell\|,
\end{equation}
where we used \eqref{eqA:dual_Nsup} and \eqref{eqA:dual_identification}, and where the constant is independent from $K$. By Fatou's lemma, we obtain $\|\NN f\|_\pp\siml \|\ell\|$. Finally, now that we know that $f\in \N^\pp(\Omega)$, we can extend \eqref{eqA:dual_identification} to all $g\in \A^p(\Omega)$; indeed, it is easy to prove that $L^1_c(\Omega,m)$ is dense in $A^p(\Omega)$, and the continuity of the right hand sign is given by \eqref{eqA:dual_prodNA}. This implies simultaneously that our identification is in fact a bijection, and that $\|\ell\|\siml\|\NN f\|_\pp$. Piecing together all that information, we obtain that $\A^p(\Omega)'$ is homeomorphic to $\N^\pp(\Omega)$.
\end{proof}

\begin{theorem}\label{thA:dual_Asup}
Let $1<p<\infty$. We have
\begin{equation}\label{eqA:dual_Asup}
\|\AA g\|_{p'} 
\leq C \sup_{\|\NN f\|_p\leq 1} \int_\Omega fg\frac{\dm}{\rho\delta} \for g\in \A^{p'}(\Omega)
\end{equation}
and
\begin{equation}\label{eqA:dual_Atsup}
\|\AAt g\|_{p'} 
\leq C \sup_{\|\NNt f\|_p\leq 1} \int_\Omega fg\frac{\dm}{\rho\delta} \for g\in \At^{p'}(\Omega),
\end{equation}
where $C>0$ depends only on $p$, $C_m'$ and $C_\mu$.
\end{theorem}

\begin{proof}
Let us begin by proving \eqref{eqA:dual_Asup}. As before, it is enough to consider $g\geq 0$ in $\A^\pp(\Omega)$. The proof uses the same ideas as in the one of \eqref{eqA:dual_Nsup}, although being much easier. By duality,
\begin{equation}
\|\AA g\|_{p'} = \sup_{\|\phi\|_p\leq 1} \int_\dOmega (\AA g) \phi \dmu.
\end{equation}
Take $\phi\geq 0$ in $L^p(\dOmega,\mu)$ such that $\|\phi\|_p\leq 1$. Using Fubini's theorem,

\begin{align*}
\int_\dOmega (\AA g)\phi \dmu
&=\int_\dOmega \phi(\xi)\int_{\gamma(\xi)} g(x) \frac{\dm(x)}{m(B_x)} \dmu(\xi)\\
&=\int_\Omega g(x)\unsur{m(B_x)}\int_{2B_x\cap\dOmega} \phi(\xi)\dmu(\xi) \dm(x)\\
&=\int_\Omega fg \frac{\dm}{\rho\delta},
\end{align*}
where $f(x):=\fint_{2B_x\cap\dOmega} \phi\dmu$ for $x\in\Omega$. Since

\begin{equation*}
\NN f(\xi)
= \sup_{x\in\gamma(\xi)} \fint_{2B_x\cap\dOmega} \phi\dmu
\siml \MM_\mu\phi (\xi) \for \xi\in\dOmega,
\end{equation*}
using the doubling properties of $\mu$, we have that $\|\NN f\|_p \siml \|\MM_\mu\phi\|_p\siml \|\phi\|_p\leq 1$, giving us \eqref{eqA:dual_Asup}. As for \eqref{eqA:dual_Atsup}, it is deduced from \eqref{eqA:dual_Asup} in the exact same way that we deduced \eqref{eqA:dual_Ntsup} from \eqref{eqA:dual_Nsup} in Theorem \ref{thA:dual_A'N}.
\end{proof}

At last, we can prove the equivalence of norms under change of parameters for $\AA$:

\begin{lemma}\label{lemmaA:eqvlce_A}
Let $1<p<\infty$. The $L^p$ norms of $\AA^{(\alpha)}$ are equivalent under change of $\alpha$, i.e. for any $\alpha,\alpha'>0$,
\begin{equation*}
\|\AA^{(\alpha)}g\|_p \leq C_{\alpha,\alpha'} \|\AA^{(\alpha')}g\|_p,
\end{equation*}
for any $g$ measurable, where $C_{\alpha,\alpha'}>0$ depends only on $p$, $C_\mu$, $C_m'$, $\alpha$ and $\alpha'$. Likewise, the $L^p$ norms of $\AAt^{(\alpha,\lambda)}$ are equivalent under change of $\alpha$ and $\lambda$, i.e. for any $\alpha,\alpha'>0$ and any $0<\lambda,\lambda'\leq 1/2$,
\begin{equation*}
\|\AAt^{(\alpha,\lambda)} g\|_p \leq C_{\alpha,\alpha',\lambda/\lambda'} \|\AAt^{(\alpha',\lambda')} g\|_p
\end{equation*}
for any $g\in L^2\loc(\Omega,m)$ ,where $C_{\alpha,\alpha',\lambda/\lambda'}>0$ depends only on $p$, $C_\mu$, $C_m'$, $\alpha$, $\alpha'$ and $\lambda/\lambda'$.
\end{lemma}
\begin{proof}
this is a direct consequence of \eqref{eqA:dual_Asup}, \eqref{eqA:dual_intNA} and Lemma \ref{lemmaA:eqvlce_N} for $\AA$, and \eqref{eqA:dual_Atsup}, \eqref{eqA:dual_int_NtAt} and Lemma \ref{lemmaA:eqvlce_N} for $\AAt$.
\end{proof}

\begin{definition}
Take $0<\lambda\leq 1/2$. For any measurable function $g:\Omega\fleche\R$, we define the Carleson function as 
\begin{equation*}
\CC g(\xi) := \sup_{r>0} \frac{1}{\mu(B(\xi,r))}\int_{B(\xi,r)\cap\Omega} |g| \frac{\dm}{\rho\delta}.
\end{equation*}
For any $g\in L^2\loc(\Omega,m)$, we define the modified Carleson function as
\begin{equation*}
\CCt^{(\lambda)} g(\xi) = \sup_{r>0} \frac{1}{\mu(B(\xi,r))}\int_{B(\xi,r)\cap\Omega} \pfrac{\fint_{B_x/2} |g|^2\dm} \frac{\dm(x)}{\rho(x)\delta(x)}.
\end{equation*}
We will note $\CCt=\CCt^{(1)}$.
\end{definition}

Carleson's duality theorem holds in our setting:
\begin{theorem}
For any $f,g\geq 0$ measurable,
\begin{equation}\label{eqA:carleson}
\int_\Omega fg\frac{\dm}{\rho\delta}
\leq C \int_\dOmega (\NN f)( \CC g)\dm,
\end{equation}
where $C>0$ depends only on $C_\mu$.
\end{theorem}
\begin{proof}
Take $f,g\geq 0$ measurable. Thanks to Cavalieri's principle, it is enough to show that
\begin{equation}\label{eqA:carleson:weak_carleson}
\nu\{x\in\Omega\tq f(x)>t\}
\siml \int_{\{\xi\in\dOmega\tq \NN f>t\}} \CC g\dmu,
\end{equation}
for any $t>0$, where $d\nu:=g\frac{dm}{\rho\delta}$. Let us also observe that for any $B=B(\xi,r)$ centered on $\dOmega$,
\begin{equation}\label{eqA:carleson:nuB}
\nu(16 B\cap\Omega)\siml \int_B\CC g\dmu.
\end{equation}
Indeed, for any $\zeta\in B$,
\begin{equation*}
\nu(16 B\cap\Omega)
\leq \nu\p{B(\zeta,32 r)\cap\Omega}
=\int_{B(\zeta,32 r)\cap\Omega} g\frac{\dm}{\rho\delta}
\siml \frac{\mu(B)}{\mu(B(\zeta,32 r))}\int_{B(\zeta, 32 r)\cap\Omega} g\frac{\dm}{\rho\delta}
\leq\mu(B)\CC g(\zeta),
\end{equation*}
using the doubling property of $\mu$, which give us the estimate by averaging over $\zeta\in B$.
\newline

Let us prove \eqref{eqA:carleson:weak_carleson}.  Fix some $t>0$, and write $E_t:=\{\xi\in\dOmega\tq \NN f>t\}$. If $E_t=\dOmega$, then taking $\xi_0\in\dOmega$ and $r>0$, applying \eqref{eqA:carleson:nuB} to $B(\xi_0,r)$ and letting $r$ grow to infinity give us that $\nu(\Omega)\leq \int_\dOmega \CC g\dmu$, implying \eqref{eqA:carleson:weak_carleson}. Let us assume that $E_t\varsubsetneq \dOmega$. $E_t$ is open in $\dOmega$: indeed, if $\xi\in\dOmega$ satisfies $\NN f(\xi)>t$, then there exists $x\in\gamma(\xi)$ such that $f(x)>t$; let us note $r_x=2\delta(x)-|x-\xi|>0$. Since $x\in\gamma(\zeta)$ for any $\zeta\in B(\xi,r_x)\cap\dOmega$, we have $B(\xi,r_x)\cap\dOmega\incl E_t$.
\newline

Since $E_t$ is open and different from $\dOmega$, we can apply the Whitney Covering Theorem (see  for example \cite{CW77}) which give us a countable family of balls $B_k=B(\xi_k,r_k)$, for ${k\in\N}$, centered on $\dOmega$, and a constant $C>0$ such that
\begin{enumerate}[(i)]
\item $\sum_{k\in\N} \UN_{B_k\cap\dOmega}\leq C$,
\item $E_t=\bigcup_{k\in\N} B_k\cap\dOmega$,
\item $4B_k\cap\dOmega\priv E_t\neq\emptyset$.
\end{enumerate}
Now we will prove that
\begin{equation}\label{eqA:carleson:incl}
\{x\in\Omega\tq f(x)>t\} \incl \bigcup_{k\in\N} 16B_k.
\end{equation}
Take $x\in\Omega$ such that $f(x)>t$, and $\xi\in\dOmega$ such that $|x-\xi|=\delta(x)$. Then $\xi\in E_t$, so there exists $k\in\N$ such that $\xi\in B_k$; thanks to $(iii)$, we can take $\zeta\in 4B_k\cap\dOmega\priv E_t$. By choice of $x$ we have that $2B_x\cap\dOmega\incl E_t$, so $\zeta\notin 2B_x$; thus, 
\begin{equation*}
2\delta(x)
\leq |\zeta-\xi|
\leq |\zeta-\xi_k|+|\xi_k-\xi|
< 8r_k,
\end{equation*}
so $|x-\xi_k|<\delta(x)+r_k<16r_k$, proving \eqref{eqA:carleson:incl}. Then,
\begin{equation*}
\nu\{x\in\Omega\tq f(x)>t\}
\leq \sum_{k\in\N} \nu(16B_k)
\siml \sum_{k\in\N} \int_{B_k} \CC g\dmu
\siml \int_{E_t} \CC g\dmu,
\end{equation*}
using the inclusion \eqref{eqA:carleson:incl} for the first estimate, \eqref{eqA:carleson:nuB} for the  second one and properties $(i)$ and $(ii)$ of the Whitney decomposition for the last one. This concludes the proof.
\end{proof}

\begin{theorem}\label{thA:eqvlce_CA}
Take $1<p<\infty$. We have
\begin{equation}\label{eqA:dual_AeqC}
C^{-1}\|\AA g\|_p\leq \|\CC g\|_p \leq C\|\AA g\|_p,
\end{equation}
for any $g:\Omega\to\R$ measurable, and
\begin{equation}\label{eqA:dual_AtEqCt}
C^{-1}\|\AAt g\|_p\leq \|\CCt g\|_p \leq C\|\AAt g\|_p,
\end{equation}
for any $g\in L^2\loc(\Omega,m)$, where $C>0$ depends only on $p$, $C_\mu$ and $C_m'$.
\end{theorem}

\begin{proof}
 Let us begin by proving \eqref{eqA:dual_AeqC}. By definition of $\AA$ and $\CC$, it is enough to prove the result for $g\geq 0$.
\newline

First let's prove that $\|\AA g\|_p\siml\|\CC g\|_p$. Take $g\geq 0$ measurable such that $\|\CC g\|_p$ is finite. This implies that $g\in L^1\loc(\Omega,m)$; take a compact $K$ in $\Omega$, and note $g_K:=g\UN_K$. Then $g_K\in \A^p(\Omega)$ and

\begin{align*}
\|\AA g_K\|_p
&\siml \sup_{\|\NN f\|_\pp\leq 1} \int_\Omega fg_K\frac{\dm}{\rho\delta}\\
&\siml \sup_{\|\NN f\|_\pp\leq 1} \|\NN f\|_\pp \|\CC g\|_p
\siml \|\CC g\|_p,
\end{align*}
where we used \eqref{eqA:dual_Asup} and Carleson's duality theorem \eqref{eqA:carleson}, and where the constants are independent from $K$. By Fatou's lemma, we obtain the result.
\newline

Now let's show the converse estimate. Take $g\geq 0$ measurable such that $\|\AA g\|_p$ is finite. For any ball $B$ centered on $\dOmega$, we compute similarly as in the proof of \eqref{eqA:dual_intA} and get

\begin{equation*}
\int_B g\frac{\dm}{\rho\delta}
\simeq \int_{2B\cap\dOmega} \AA g\dmu,
\end{equation*}
which implies that $\CC g(\xi) \siml \MM_\mu\AA g(\xi)$ for any $\xi\in\dOmega$. Thus $\|\CC g\|_p \siml \|\MM_\mu\AA g\|_p \siml \|\AA g\|_p$, which give us \eqref{eqA:dual_AeqC}. As for \eqref{eqA:dual_AtEqCt}, it is a direct consequence of \eqref{eqA:dual_AeqC}.

\end{proof}
Obviously, the equivalence holds for $\CCt^{(\lambda)}$ as well; as a direct consequence, the $L^p$ norms of $\CC^{(\lambda)}$ are equivalent under change of $\lambda$.

\addcontentsline{toc}{section}{Acknowledgements}
\section*{Acknowledgements}
This work has been carried out thanks to the grant "Programme blanc GS Mathématiques" of the \'Ecole doctorale Mathématiques Hadamard, at Université Paris-Saclay. The author would like to thank his advisor, Joseph Feneuil, for his help while establishing these results, and during the writing and proofreading of this article.

\addcontentsline{toc}{section}{References}
\printbibliography

\end{document}